\documentclass{amsart}

\usepackage{amssymb,graphicx,color}
\usepackage[colorlinks=true,citecolor=black,linkcolor=black,urlcolor=blue]{hyperref}

\def\cE{\mathcal{E}}

\newcommand{\beqa}{\begin{eqnarray}}
\newcommand{\eeqa}{\end{eqnarray}}

\def\cS{\mathcal{S}}
\def\Fsp{F_{s}^{(p)}}
\def\wC{\widehat{C}}
\def\Vexc{V_{\mathrm{exc}}}

\theoremstyle{plain}

\newtheorem{lem}{Lemma}[section]
\newtheorem{coro}{Corollary}[section]
\newtheorem{claim}{Claim}[section]
\newtheorem{prop}{Proposition}[section]

\theoremstyle{definition}

\theoremstyle{remark}
\newtheorem{rk}{Remark}[section]

\title{Asymptotic expansion of the multi-orientable random tensor model}

\author{\'Eric Fusy and Adrian Tanasa}

\date{\today}

\begin{document}

\begin{abstract}
Three-dimensional random tensor models are a natural generalization 
of the celebrated matrix models. 
The associated tensor graphs, or 3D maps, can be classified with respect to a particular integer or half-integer, the degree of the respective graph.
In this paper we analyze the general term of the asymptotic expansion in $N$, the size of the tensor, 
of a particular random tensor model, the multi-orientable tensor model.
We perform their enumeration and we establish which are the dominant configurations of a given degree. 
\end{abstract}

\maketitle

\smallskip
{\small \noindent \textbf{Keywords.} 3D maps, 4-regular maps, Eulerian orientations, schemes}

\section{Introduction and motivation}

Random tensor models (see \cite{rev-riv} for a recent review) generalize in dimension three (and higher) the celebrated matrix models (see, for example, \cite{df} for a review). Indeed, in the same way matrix models are related to combinatorial maps~\cite{zvonkine}, tensor models in dimension three are related to tensor graphs or 3D maps. 
A certain case of random tensor models, the so-called colored tensor models, have been intensively studied in the recent years (see \cite{GR} for a review). 
The graphs associated to these models are regular edge-colored graphs.
An important result is the 
asymptotic expansion in the 
limit $N\to\infty$ ($N$ being the size of the tensor), an expansion which was obtained in \cite{largeN}. The role played by the genus in the 2D case is played here by a distinct integer called the {\it degree}. The dominant graphs of this asymptotic expansion are the so-called melonic graphs, which correspond to particular triangulations of the three-dimensional sphere ${\cS}^3$. Let us also mention here that a universality result generalizing
 matrix universality was obtained in the tensor case in \cite{univ}.

A particularly interesting combinatorial approach for the study of these colored graphs was proposed recently by Gur\u{a}u and Schaeffer in \cite{GS}, where they  analyze in detail the structure of colored graphs of fixed degree and perform exact and asymptotic enumeration. This analysis relies on the reduction of colored graphs to some terminal forms, called {\it schemes}. An important result proven in \cite{GS} is that the number of schemes of a given degree is finite (while the number of graphs of a given degree is infinite). 
%This analysis allowed to implement the so-called {\it double scaling limit} for colored models, which is a particularly important mechanism for matrix models (see again \cite{df}), and which makes it possible to take, in a correlated way, the double limit $N\to\infty$ and $z\to z_c$ where $z$ is a variable counting the number of vertices of the graph and $z_c$ is some critical point of the generating function of schemes of a given degree. 
%Let us also mention that the double scaling mechanism was implemented, for a different type of tensor colored model, in \cite{DGR}, using quantum field theoretical-inspired methods.

%\medskip

Nevertheless, a certain drawback of colored tensor models is that a large number of tensor graphs is discarded by the very definition of the model. Thus, a different type of model was initially proposed in \cite{mo}, the 3D multi-orientable (MO) tensor model. This model is related to tensor graphs which correspond to 3D maps with a particular Eulerian orientation.  
The set of MO tensor graphs contains as a strict subset the set of colored tensor graphs (in 3D). 
The  asymptotic expansion in the limit $N\to\infty$ for the MO tensor model was studied in \cite{DRT}, where it was shown that the same class of tensor graphs, the melonic ones, are the dominant graphs in this limit. The sub-dominant term of this expansion was then studied in detail in \cite{RT}.

%\medskip

In this paper we implement a Gur\u{a}u-Schaeffer analysis for the MO random tensor model. We 
investigate in detail the general term of the asymptotic expansion in the limit $N\to\infty$. As in the colored case, this is done by defining appropriate terminal forms, the schemes. Nevertheless, our analysis is somehow more involved from a combinatorial point of view, since, as already mentioned above, a larger class of 3D maps has to be taken into consideration. 
%As a consequence, the class of schemes we define is significantly larger than the ones considered by Gur\u{a}u and Schaeffer.
%Moreover, the double scaling limit mechanisms is implemented here for the MO case.
Also an important difference with respect to the colored model, which only allows for integer degrees, is that the MO model allows for both half-odd-integer and integer degrees. This leads to the the fact that the dominant schemes are different from the ones identified in \cite{GS} for the colored model (interestingly, in both cases, dominant schemes are naturally associated to rooted binary trees).

Let us also mention that the analysis of this paper may further allow for the implementation of the the so-called {\it double scaling limit} for the MO tensor model. This is a particularly important mechanism for matrix models (see again \cite{df}), making it possible to take, in a correlated way, the double limit $N\to\infty$ and $z\to z_c$ where $z$ is a variable counting the number of vertices of the graph and $z_c$ is some critical point of the generating function of schemes of a given degree.

\section{Preliminaries}

In this section we recall the main definitions related to MO tensor graphs.

\subsection{Multi-orientable tensor graphs}
A \emph{map} (also called a fat-graph) is a graph (possibly with loops and multiple edges, possibly disconnected) such that at each vertex $v$ 
the cyclic order of the $d$ incident half-edges ($d$ being the degree of $v$) in clockwise (cw) order around $v$ is specified.
A \emph{corner} of a map is defined as a sector between two consecutive half-edges around a vertex,
so that a vertex of degree $d$ has $d$ incident corners.   
A \emph{4-regular map} is a map where all vertices have degree $4$. A \emph{multi-orientable tensor graph}, shortly called MO-graph hereafter, is a 4-regular map where each half-edge carries a sign, $+$ or $-$, such that each edge has its two half-edges of opposite signs, and the two half-edges at each corner also have opposite signs. In addition, for convenience, the half-edges at each vertex
are turned into 3 parallel strands, see Figure~\ref{fig:exemple_MO_graph} for an example. 

The strand in the middle is called \emph{internal}, the two other ones are called \emph{external}.
An external strand is called \emph{left} if it is on the left side of a positive half-edge 
or on the right side of a negative half-edge; an external strand is called \emph{right} if it is on the right side of a positive half-edge 
or on the left side of a negative half-edge. 
A \emph{face} of an MO-graph is a closed walk formed by a closed (cyclic) sequence of strands. 
External faces (faces formed by external strands) are the classical faces of the 4-regular map, while \emph{internal faces} (faces formed by internal strands), also called \emph{straight faces} thereafter,  are not faces of the 4-regular map.  Note also that external faces are either made completely of 
left external strands, or are made completely of right external strands; accordingly external faces
are called either left or right. 
We finally define a \emph{rooted MO-graph} as a connected MO-graph with a marked edge, which is convenient  (as in the combinatorial study of maps) to avoid symmetry issues (indeed, as 
for maps, MO graphs are \emph{unlabelled}).

\subsection{The degree of an MO-graph}\label{sec:degree}
The \emph{degree} of an MO-graph $G$ is the quantity $\delta\in\tfrac12\mathbb{Z}_+$ defined by
\begin{equation}
2\delta=6c+3V-2F,
\end{equation}
where $c,V,F$ are respectively the numbers of connected components, vertices, and faces (including internal faces) of $G$. For $G_1$ and $G_2$ two MO-graphs, denote by $G_1+G_2$ the MO-graph
made of disconnected copies of $G_1$ and $G_2$. Note that the quantities $c,V,F$ for $G$
result from adding the respective quantities from $G_1$ and $G_2$. Hence we have:
\begin{claim}\label{claim:disconnected}
For two MO-graphs $G_1$ and $G_2$, the degree of $G_1+G_2$ is the sum of the degrees of $G_1$ and $G_2$. In particular, the degree of an MO-graph is the sum of the degrees of its connected components. 
\end{claim}

That $\delta\in\tfrac12\mathbb{Z}_+$ follows from the following observation~\cite{mo}:
$G$ gives rise to three 4-regular maps $G_{\ell,r}$ (resp. $G_{\ell,s}$, $G_{r,s}$), called
the \emph{jackets} of $G$, which are obtained from $G$  by deleting straight faces (resp. deleting right faces, deleting left faces), see Figure~\ref{fig:jacketex} for an example. Note that $G_{\ell,r}$ is an orientable map, while $G_{\ell,s}$ and $G_{r,s}$ are typically only locally orientable (in a usual ribbon representation, edges have twists). Let $g_{\ell,r}$, $g_{\ell,s}$
and $g_{r,s}$ be the respective genera of $G_{\ell,r}$, $G_{\ell,s}$, $G_{r,s}$ (since $G_{\ell,r}$ is orientable, $g_{\ell,r}\in\mathbb{Z}_+$, while $g_{\ell,s}$ and $g_{r,s}$ are in $\tfrac12\mathbb{Z}_+$). 
 Let $F_{\ell}$, $F_r$, 
$F_s$ be the numbers of left faces, right faces, and straight faces of $G$. Denoting by $E$
the number of edges of $G$, the Euler relation gives
$V-E+(F_{\ell}+F_r)=2c-2g_{\ell,r}$,  $V-E+(F_{\ell}+F_s)=2c-2g_{\ell,s}$, $V-E+(F_{r}+F_s)=2c-2g_{r,s}$.
Since $F=F_{\ell}+F_{r}+F_{s}$ and $E=2V$ (because the map is 4-regular) we conclude that 
$$
\delta=g_{\ell,r}+g_{\ell,s}+g_{r,s}.
$$
We call $g:=g_{\ell,r}\in\mathbb{Z}_+$ the \emph{(canonical) genus} of $G$. An MO-graph $G$ is 
called \emph{planar} if $g=0$. 

\begin{figure}
\centering
\def\svgwidth{0.5\columnwidth}
\includegraphics[width=8cm]{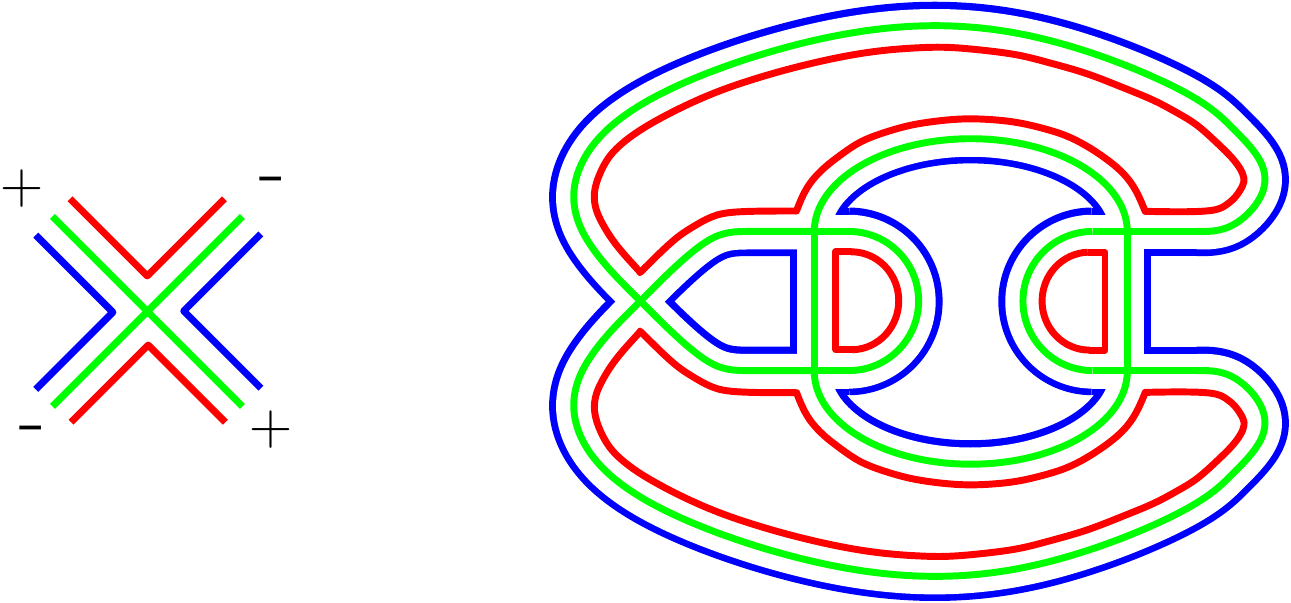}
\caption{An MO-graph, with the 3 types of strands (left strands in blue, straight strands in green, right strands in red).}
\label{fig:exemple_MO_graph}
\end{figure}

\begin{figure}
\centering
\def\svgwidth{0.65\columnwidth}
\includegraphics[width=10cm]{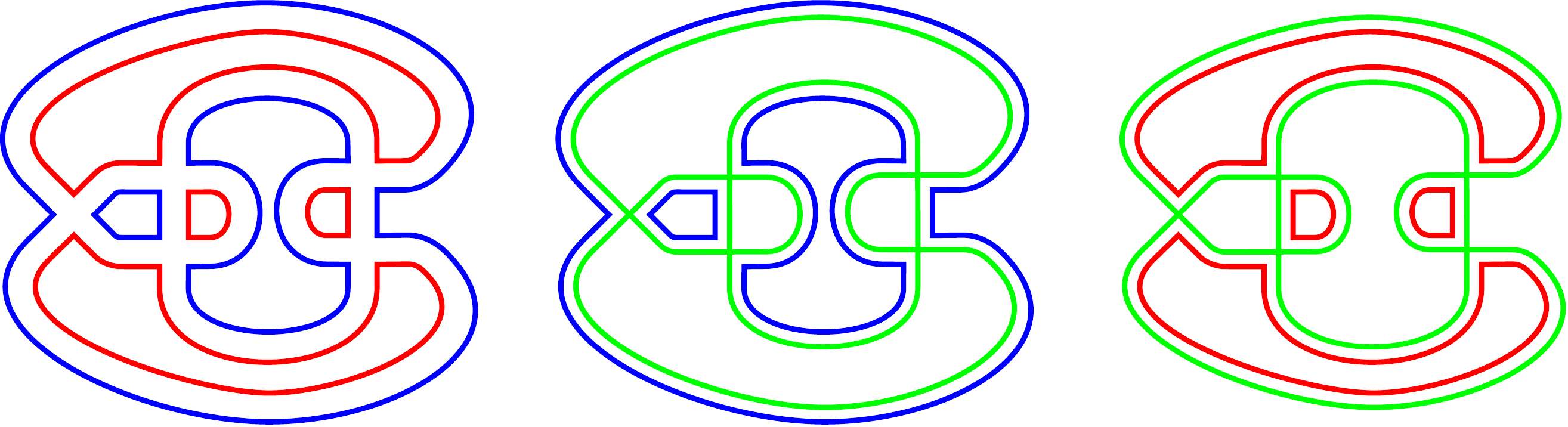}
\caption{The three jackets of the MO graph in Figure~\ref{fig:exemple_MO_graph}.}
\label{fig:jacketex}
\end{figure}

\subsection{MO-graphs as oriented 4-regular maps.}
Seeing edges as directed from the positive to the negative half-edge, 
MO-graphs may be  seen as 4-regular maps where the edges are directed such that each vertex has two ingoing and two outgoing half-edges, and the two ingoing (resp. two outgoing) half-edges are opposite, see Figure~\ref{fig:exemple_MO_graph_oriented} for an example. We call \emph{admissible orientations} such orientations of 4-regular maps. 
This point of view of MO-graphs as oriented 4-regular maps is quite convenient and will be adopted most of the time in the rest of the article (except
when we need to locally follow the strand structure).
Note that  the left faces are directed counterclockwise (ccw), and the right faces are directed cw; and the straight faces are closed walks that alternate in direction at each vertex they pass by.

\begin{figure}
\centering
\def\svgwidth{0.5\columnwidth}
\includegraphics[width=8cm]{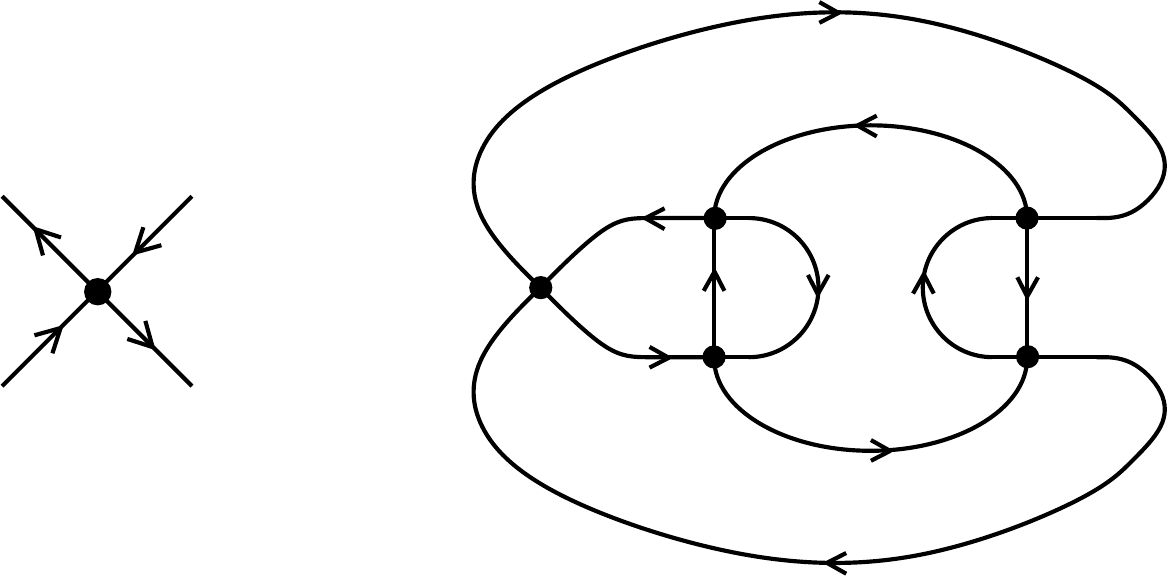}
\caption{The MO-graph shown in Figure~\ref{fig:exemple_MO_graph}, represented as an oriented 4-regular map.}
\label{fig:exemple_MO_graph_oriented}
\end{figure}

\begin{rk}
Since straight faces alternate at each vertex they pass by, straight faces have even length. 
\end{rk}

Note also that admissible orientations are completely \emph{characterized} by the property that the faces of the 4-regular map (not including the straight faces) are either cw or ccw; this is also 
equivalent to coloring the faces of the map in two possible colors, say red or blue,
 such that at each edge there is a blue face on one side and a red face on the other side (this 
 is also directly visible given the definition with strands, where right faces are red and left faces are blue, see Figure~\ref{fig:exemple_MO_graph}). In other words, a 4-regular map can be endowed with an admissible orientation
iff the dual of every of its connected components is a bipartite quadrangulation.
If this is the case, and if there is a marked directed edge in each connected component,
then there is a unique admissible orientation that fits with the prescribed edge directions 
(existence follows from the above discussion, and there is clearly a unique way 
to propagate the directions starting from the marked directed edges).  

Recall that, in the context of maps, a \emph{rooted map} means a connected map with a marked directed edge.  
The above paragraph yields the following statement, which we think is worth mentioning even if it will not be needed to perform the combinatorial study (scheme extraction and associated analysis) of MO-graphs:

\begin{rk}\label{rk:bij_4regular}
Rooted MO-graphs having genus $g$ and $m$ vertices may be identified with rooted 4-regular maps having genus $g$ and $m$ vertices, and with the property that the dual rooted quadrangulation is bipartite (these are themselves well-known to be in bijection with rooted connected maps having genus $g$ and $m$ edges). In genus $0$ (planar case), all rooted 4-regular maps have this property. 
\end{rk}

We conclude this subsection with a simple lemma relating the degree and the genus to 
the number of straight faces:

\begin{lem}\label{lem:straight_faces}
Let $G$ be a connected MO-graph. Let $V$ be the number of vertices, $F_s$ the number of straight faces, and for each $p\geq 1$,  $F_s^{(p)}$ the number of straight faces of length $2p$ in $G$.  Then the quantity
$\Lambda:=\sum_{p\geq 1}(p-2)\cdot F_s^{(p)}+2=V-2F_s+2$ satisfies
$$
\Lambda=2\delta-4g,
$$
where $\delta$ and $g$ are respectively the degree and the genus of $G$. Hence $\Lambda\leq 2\delta$. 
\end{lem}
\begin{proof}
Since the number $E$ of  edges is twice the number of vertices, we have $\sum_{p\geq 1}2pF_s^{(p)}=E=2V$, so that $\sum_{p\geq 1}pF_s^{(p)}=V$, hence $\Lambda=V-2F_s+2$. 
Let $F_{\ell}$, $F_r$ be the numbers of left faces and right faces of $G$. 
We have seen in Section~\ref{sec:degree}  that $\delta=g_{\ell,r}+g_{\ell,s}+g_{r,s}$, where $g=g_{\ell,r}$ and  
$$
-V+(F_{\ell}+F_r)=2-2g_{\ell,r},\ \  -V+(F_{\ell}+F_s)=2-2g_{\ell,s},\ \  -V+(F_{r}+F_s)=2-2g_{r,s},\ \  
$$
Substracting the sum of the two last equalities from the first one, we obtain:
$$
V-2F_s=-2+2g_{\ell,s}+2g_{r,s}-2g_{\ell,r}=-2+2\delta-4g,
$$
so that $\Lambda=2\delta-4g$. 
\end{proof}

\begin{rk}\label{rk:doodles}
For $g=0$, according to Remark~\ref{rk:bij_4regular}, rooted MO-graphs correspond to rooted  4-regular planar maps. Seeing such a map as the planar projection of an entangled link that lives in the 3D space (vertices of the map correspond to crossings, where the under/over information is omitted),  $F_s$ is classically interpreted as the number of \emph{knot-components}. 
Lemma~\ref{lem:straight_faces} gives $F_s=V/2+1-\delta$. Since $F_s\geq 1$ and $\delta\geq 0$, the extremal cases are: (1) $\delta=0$, in which case $F_s=V/2+1$; (2) $F_s=1$, in which case $\delta=V/2$
(the MO-graph of Figure~\ref{fig:exemple_MO_graph_oriented}, having $V=5$, $F_s=2$, and thus $\delta=3/2$, is intermediate). While the case $\delta=0$ is combinatorially well-understood (as will be recalled in Section~\ref{sec:extract_melon_free_core}), the case $F_s=1$ is notoriously difficult: 
in that case, 4-regular planar maps are projected diagrams of one-component links (i.e., knot diagrams), see~\cite{ScZi} and references therein.   

We will go back to this remark
in Section~\ref{sec:gen_funct} (on generating functions), where we will establish, for each fixed $\delta\in\tfrac12\mathbb{Z}_+$, the asymptotic enumeration of rooted 4-regular planar maps with $2n$ vertices and with $n+1-\delta$ knot-components, as $n\in\delta+\mathbb{Z}$ goes to infinity.  
\end{rk}

%\subsection{The case $\delta=0$: melonic graphs}

\subsection{Regular colored graphs for $D=3$ as a subfamily of MO-graphs}\label{sec:regular_subfamily}
As recalled here from~\cite{mo}, MO-graphs form a superfamily of the well studied \emph{regular colored graphs} of dimension $3$.  
For $D\geq 2$ a \emph{regular colored graph  of dimension $D$}   
 is defined as a $(D+1)$-regular bipartite graph $G$ (vertices are either
black or white) where the edges have a color in $\{0,\ldots,D\}$, such that at each vertex the $D+1$ incident   edges have different colors.  A \emph{rooted} colored graph is a connected colored graph with a marked edge of color $0$.   For $0\leq i<j\leq D$, a \emph{face} of type $(i,j)$ in a colored graph
of dimension $D$ is a cycle made of edges that alternate colors in $\{i,j\}$; let $F_{i,j}$ be the number of faces of type $(i,j)$ in $G$. Let $c$ be the number of connected components, $V=2k$ the number of vertices and 
$F$ the total number of faces of $G$, i.e., $F=\sum_{0\leq i<j\leq D}F_{i,j}$. 
 The \emph{degree} of $G$ is the integer $\delta$ given by 
\begin{equation}
\delta=\frac12D(D-1)k+c\cdot D-F.
\end{equation}

\begin{figure}
\begin{center}
\includegraphics[width=12cm]{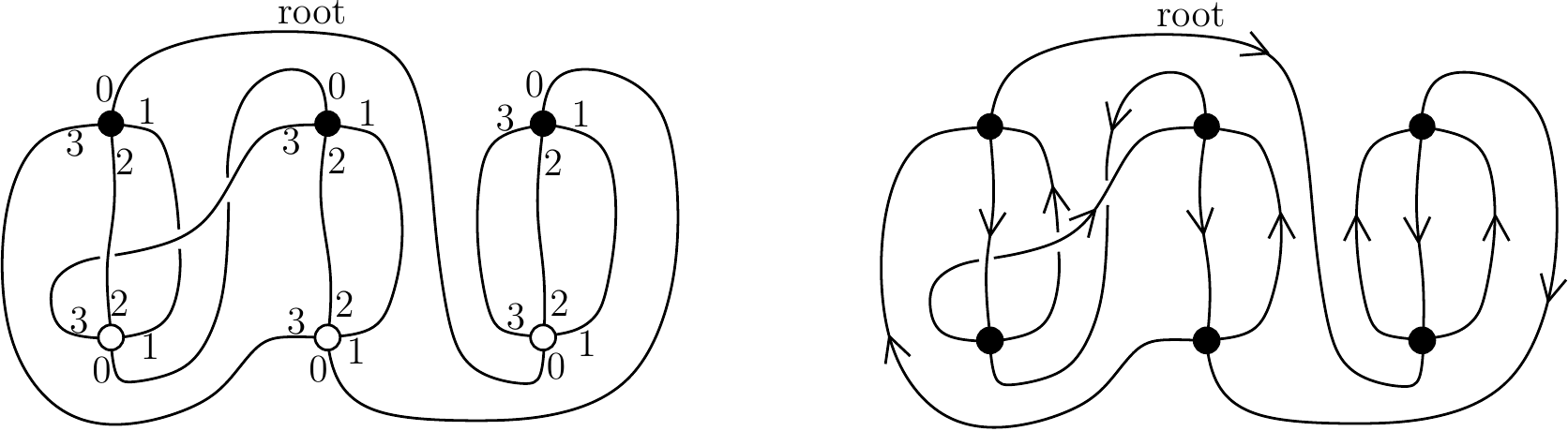}
\end{center}
\caption{Left: a rooted colored graph in $3$ dimensions. Right: the induced rooted MO-graph.}
\label{fig:colored}
\end{figure}

Note that a colored graph has a canonical realization as a map, where the 
edge colors in cw (resp. ccw) order around black (resp. white) vertices
are $(0,1,\ldots,D)$. Let $G$ be a rooted colored graph of dimension $3$ that is canonically embedded. Orienting the edges of even color from the black to the white extremity and the edges of odd color from the white to the black extremity, we obtain a rooted MO-graph $\tilde{G}$,
see Figure~\ref{fig:colored} for an example. Clearly this gives an injective mapping, since 
there is a unique way (when possible) to propagate the edge colors starting from the root-edge. 
Hence, rooted colored graphs of dimension $3$ form a subfamily of rooted MO-graphs. In addition 
$$
F_{r}(\tilde{G})=F_{0,1}(G)+F_{2,3}(G),\ \ F_{\ell}(\tilde{G})=F_{1,2}(G)+F_{0,3}(G).\ \ F_{s}(\tilde{G})=F_{0,2}(G)+F_{1,3}(G),
$$
so that $\tilde{G}$ and $G$ have the same total number of faces. Hence the degree formula for colored graphs is consistent with the degree formula for MO-graphs. 

\section{From MO-graphs to schemes}
\subsection{Extracting the melon-free core of a rooted MO-graph}\label{sec:extract_melon_free_core}

\begin{figure}
\begin{center}
\includegraphics[width=8cm]{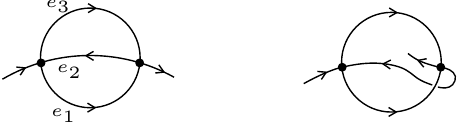}
\end{center}
\caption{Left: a melon. Right: a triple edge that does not form a melon.}
\label{fig:melon}
\end{figure}

From now on, it is convenient to consider that a ``fake'' vertex of degree $2$, called the \emph{root-vertex}, is inserted in 
the middle of the root-edge of any rooted MO-graph (so that the root-edge is turned into
two edges).
By convention also,  it is convenient to introduce the following MO-graph: the \emph{cycle-graph}
is defined as an oriented self-loop carrying no vertex. 
The cycle-graph is connected, has $V=0$, $F=3$ (one face in each
type), hence has degree $0$. In its rooted version, 
 the (rooted) \emph{cycle-graph} is made of an oriented loop incident to the root-vertex. In a (possibly  rooted) MO-graph $G$, a \emph{melon} is a triple edge $(e_1,e_2,e_3)$ such that none of the 
$3$ edges is the root-edge (if $G$ is rooted),  $(e_1,e_2)$ form a 
left face of length $2$, $(e_2,e_3)$ form a right face of length $2$, and $(e_1,e_3)$ form
a straight face of length $2$, see Figure~\ref{fig:melon}. 
Define the \emph{removal} of a melon as the operation below (where possibly $u=v$, and possibly $u$ or $v$ might be the root if the MO-graph is rooted):
\begin{center}
\includegraphics[width=8cm]{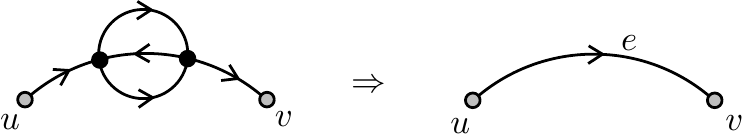}
\end{center}
The reverse operation (where $e$ is allowed to be a loop, and is allowed to be incident to the root-vertex if the MO-graph is rooted) is called the \emph{insertion} of a melon at an edge.

\begin{figure}
\begin{center}
\includegraphics[width=12cm]{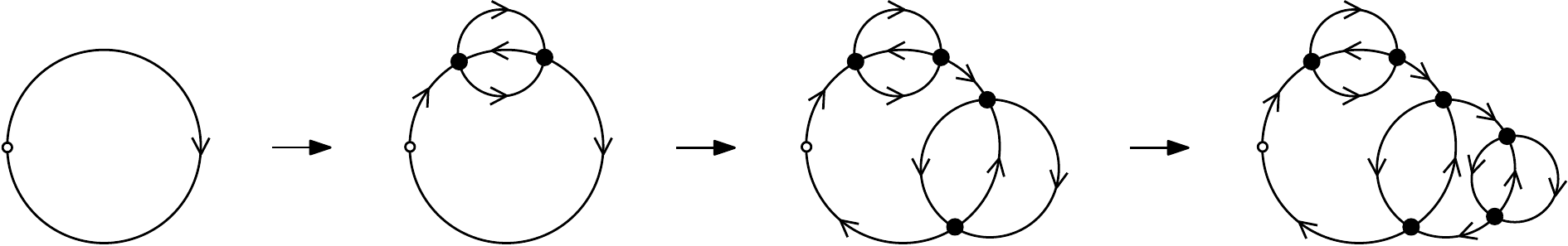}
\end{center}
\caption{A rooted melonic graph is one that can be built from the cycle-graph, by successive
insertions of melons at edges.}
\label{fig:melonic}
\end{figure}

 A rooted MO-graph $G$ is called \emph{melonic} if
it can be reduced to the rooted cycle-graph by successive removals of melons, 
see Figure~\ref{fig:melonic} for an example. 
An important remark is that, for a rooted melonic MO-graph, \emph{any} greedy sequence of melon removals terminates at the cycle-graph.  
  It is known~\cite{DRT} that rooted MO-graphs of degree $0$ are exactly the rooted melonic graphs.  For $G$ a (possibly rooted) MO-graph, $e$ an edge of $G$
(possibly a loop, possibly incident to the root-vertex), and $G'$ a rooted MO-graph,   define the operation of \emph{substituting $e$ by $G'$ in $G$} as in the following generic drawing:

\vspace{.2cm}

\begin{center}
\includegraphics[width=10cm]{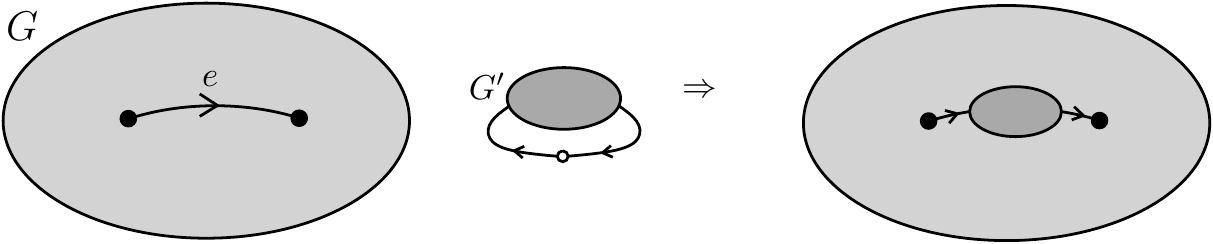}
\end{center}

\vspace{.2cm}

We have the following bijective statement, which is the counterpart for MO-graphs 
of~\cite[Theo.~4]{GS}:
\begin{prop}\label{prop:melon_free}
Each rooted MO-graph $G$ is uniquely obtained as a rooted melon-free MO-graph $H$ ---called the \emph{melon-free core} of $G$--- where
each edge $e$ is substituted by a rooted melonic MO-graph (possibly the rooted cycle-graph, in which case $e$ is unchanged). In addition, the degree of $G$ equals the degree of $H$.  
\end{prop}
\begin{proof}
%Detailed arguments for colored graphs are given in~\cite{GS}.
%; we remain more sketchy here. 
Let $H$ be the rooted melon-free graph obtained from $G$ after performing a maximal greedy sequence of melon removals. Conversely, $G$ is obtained from $H$ where a sequence of melon insertions
is performed. Hence $G$ is equal to $H$ where each edge $e$ is substituted by a rooted melonic 
graph $G_e$. This gives the existence of a melon-free core. Uniqueness of the melon-free core is given  by the observation that any other maximal greedy sequence of melon removals starting from $G$ has to progressively shell the melonic components $G_e$, hence terminates at $H$.  
Finally, it is clear that $G$ and $H$ have the same degree, since a melon insertion preserves the degree (it preserves the number of connected components, increases the number of vertices by $2$,
and increases the number of faces by $3$). 
\end{proof}

\begin{figure}
\begin{center}
\includegraphics[width=12cm]{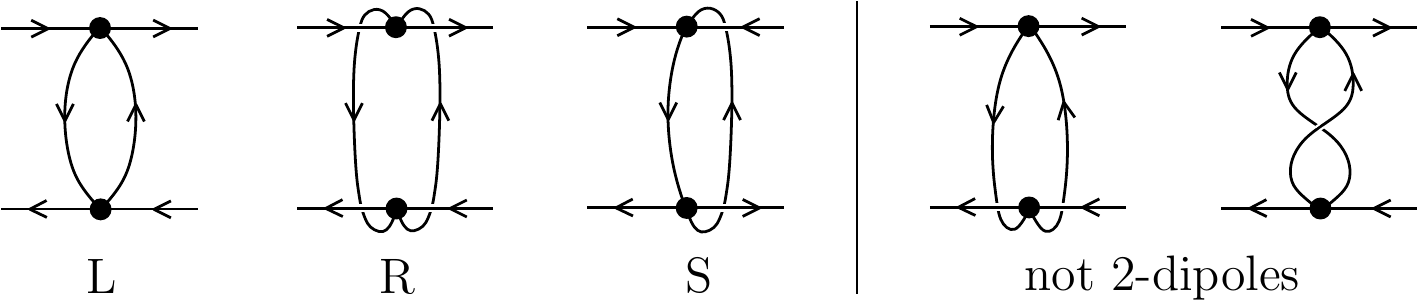}
\end{center}
\caption{Left: the three types of dipoles. Right: two examples of double edges that do not form
a dipole.}
\label{fig:dipoles}
\end{figure}

\begin{figure}
\begin{center}
\includegraphics[width=6cm]{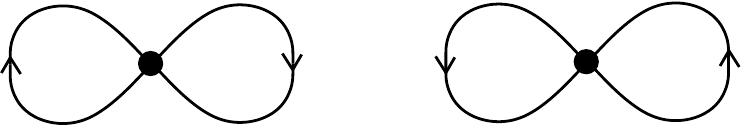}
\end{center}
\caption{Left: the clockwise infinity graph. Right: the counterclockwise infinity graph. Both have 
 two faces of length $2$, two faces of length $1$, and have degree $1/2$.}
\label{fig:infinity}
\end{figure}

\subsection{Extracting the scheme of a rooted melon-free MO-graph}
In this section we go a step further and define the process of extracting a \emph{scheme}
from a rooted melon-free MO-graph. 
%From now on  we do not need anymore to have the convention that a fake vertex of degree $2$ is placed in the middle of the root-edge.

Define a \emph{2-dipole}, shortly a dipole,  
in a (possibly rooted) MO-graph $G$ as a face of length $2$ incident to two distinct vertices, and not passing by the root if $G$ is rooted.     
Accordingly, one has three distinct types of dipoles: L, R, or S, see Figure~\ref{fig:dipoles} (left part). 
As the figure shows, each dipole has two exterior half-edges on one side and two exterior
half-edges on the other side. 
Notice also that a double edge does not necessarily delimit a dipole, as shown in Figure~\ref{fig:dipoles} (right part). 

\begin{rk}\label{rk:after_def_dipole}
A face of length $2$ is actually always incident to two distinct vertices, except in the
   MO-graphs that are made of one vertex and two loops. These two graphs, shown in Figure~\ref{fig:infinity}, are 
called the clockwise and the counterclockwise infinity graph, respectively. They have no dipole but have two faces of length $2$.  
\end{rk}
%, except for the very particular case of t), so that 
%we can assume without loss of generality that a 2-dipole is incident to two distinct vertices.

In an MO-graph $G$, define a \emph{chain} as a sequence of  dipoles $d_1\ldots,d_p$ (not passing
by the root if $G$ is rooted) 
such that for each $1\leq i<p$, $d_i$ and $d_{i+1}$ are connected by two edges involving two half-edges on the same side of $d_i$ and two half-edges on the same side of $d_{i+1}$, see  Figure~\ref{fig:chain_dipoles}. A chain is called \emph{unbroken} if all the $p$ dipoles are of the same type. 
A {\it broken chain} is a chain which is not unbroken.
A \emph{proper chain} is a chain of at least two dipoles. A proper chain is called  
  {\it maximal} if it cannot be extended into a larger proper chain. By very similar arguments as in Lemma~8 of~\cite{GS} one obtains the following result:

\begin{figure}[t!]
\begin{center}
\includegraphics[width=12cm]{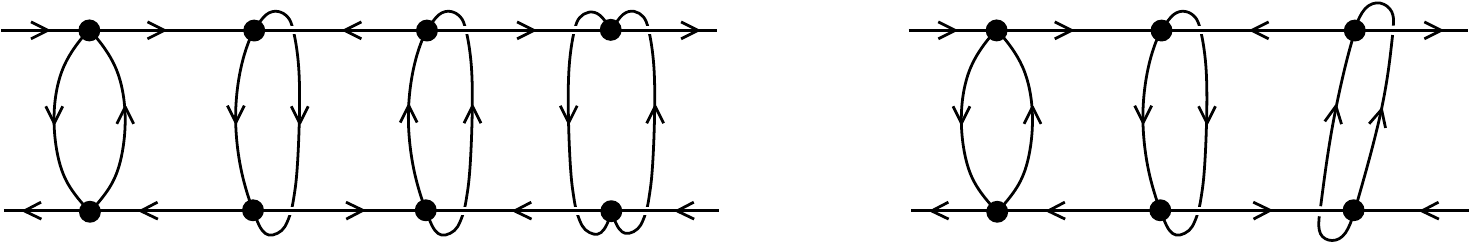}
\end{center}
\caption{Left: a chain of $4$ dipoles. Right: a sequence of 3 dipoles that do not form a chain (the 3rd element, a dipole of type $S$, is connected to the 2nd element by two half-edges on opposite sides).} 
\label{fig:chain_dipoles}
\end{figure}

\begin{claim}\label{claim:disjoint}
In a rooted MO-graph, any two maximal proper chains are vertex-disjoint. 
\end{claim}

Let $G$ be a rooted melon-free MO-graph. 
The \emph{scheme} of $G$ is the graph obtained by simultaneously replacing any maximal proper chain
of $G$ by a so-called \emph{chain-vertex}, as shown in Figure~\ref{fig:chains}, see also
Figure~\ref{fig:example_scheme} for an example.  

\begin{figure}[t!]
\begin{center}
\includegraphics[width=13cm]{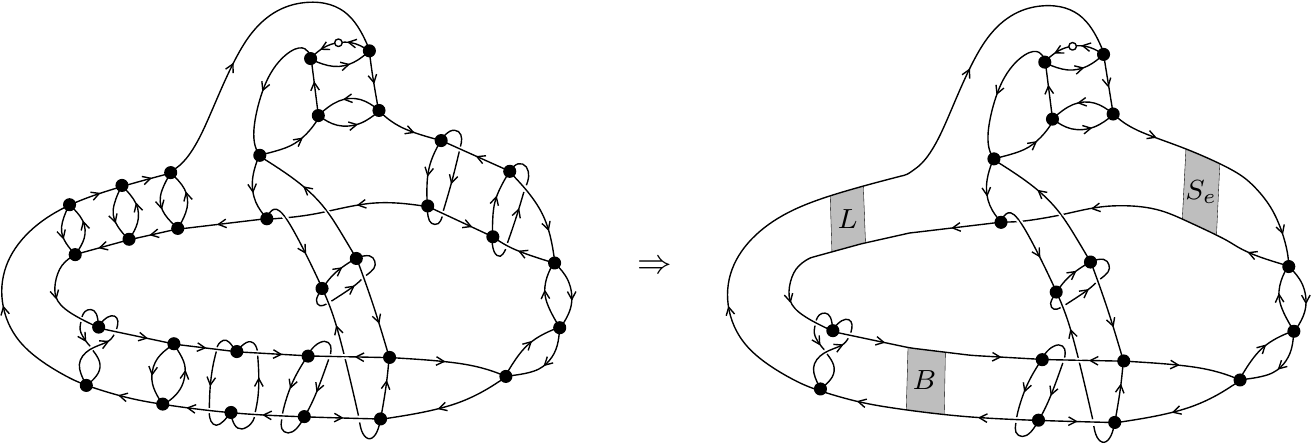}
\end{center}
\caption{Left: a rooted melon-free MO-graph. Right: the associated scheme.} 
\label{fig:example_scheme}
\end{figure}

\begin{figure}[t!]
\begin{center}
\includegraphics[width=4.5cm]{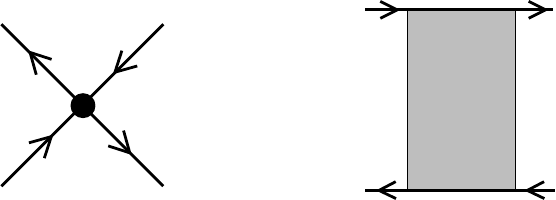}
\end{center}
\caption{The two types of vertices (standard vertex or chain-vertex, which can be labelled
by either $\{L,R,S_e,S_o,B\}$) in an MO-graph with
chain-vertices.} 
\label{fig:rules_MO_graph_chain}
\end{figure}

\begin{figure}[t!]
\begin{center}
\includegraphics[width=12cm]{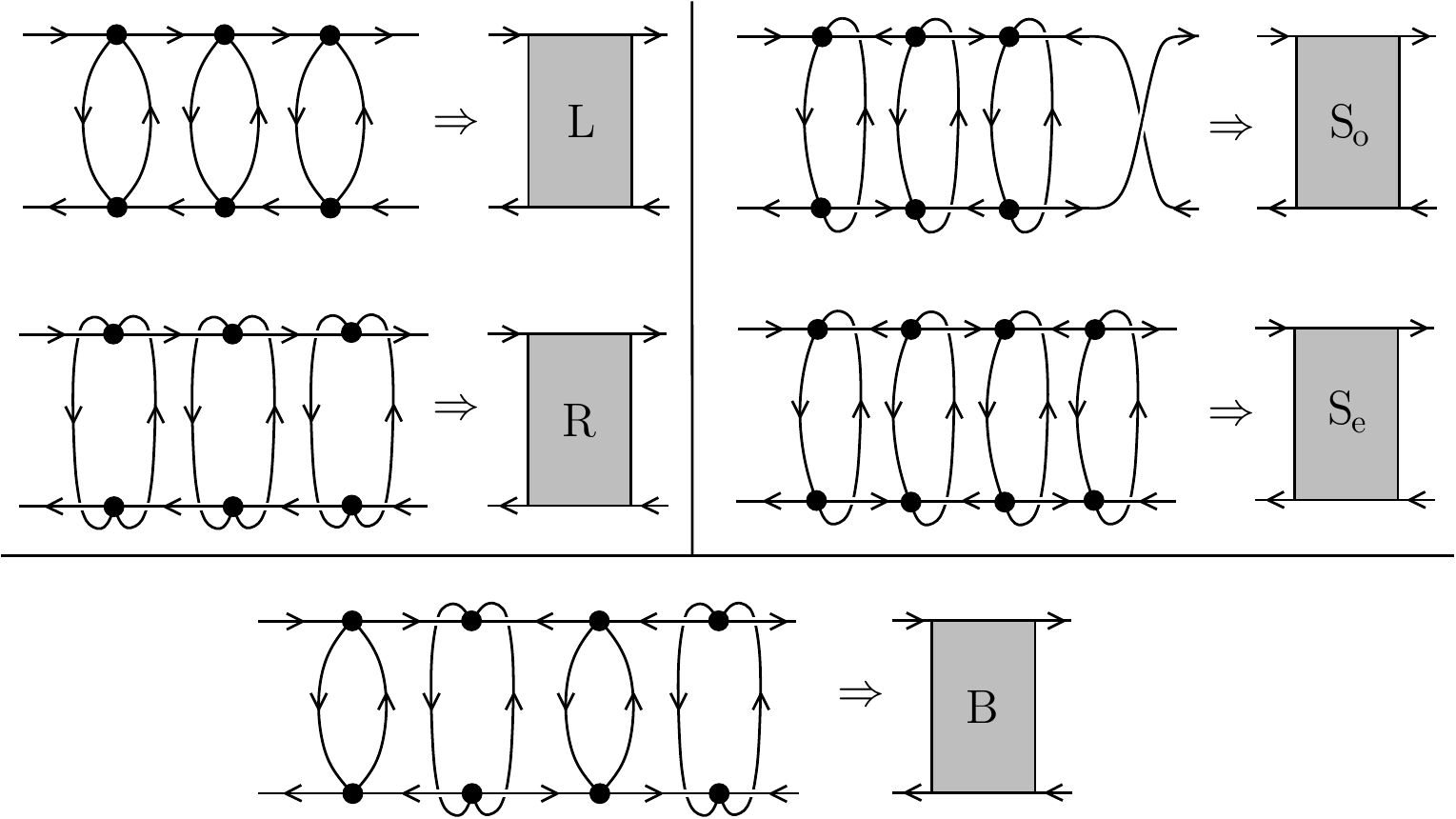}
\end{center}
\caption{Top-part: examples of the 4 types of unbroken proper chains;  Bottom-part: example
of a broken proper chain. In each case, the graphical
symbol for the chain-vertex (in the associated scheme) is shown.} 
\label{fig:chains}
\end{figure}
\begin{figure}[t!]
\begin{center}
\includegraphics[width=12cm]{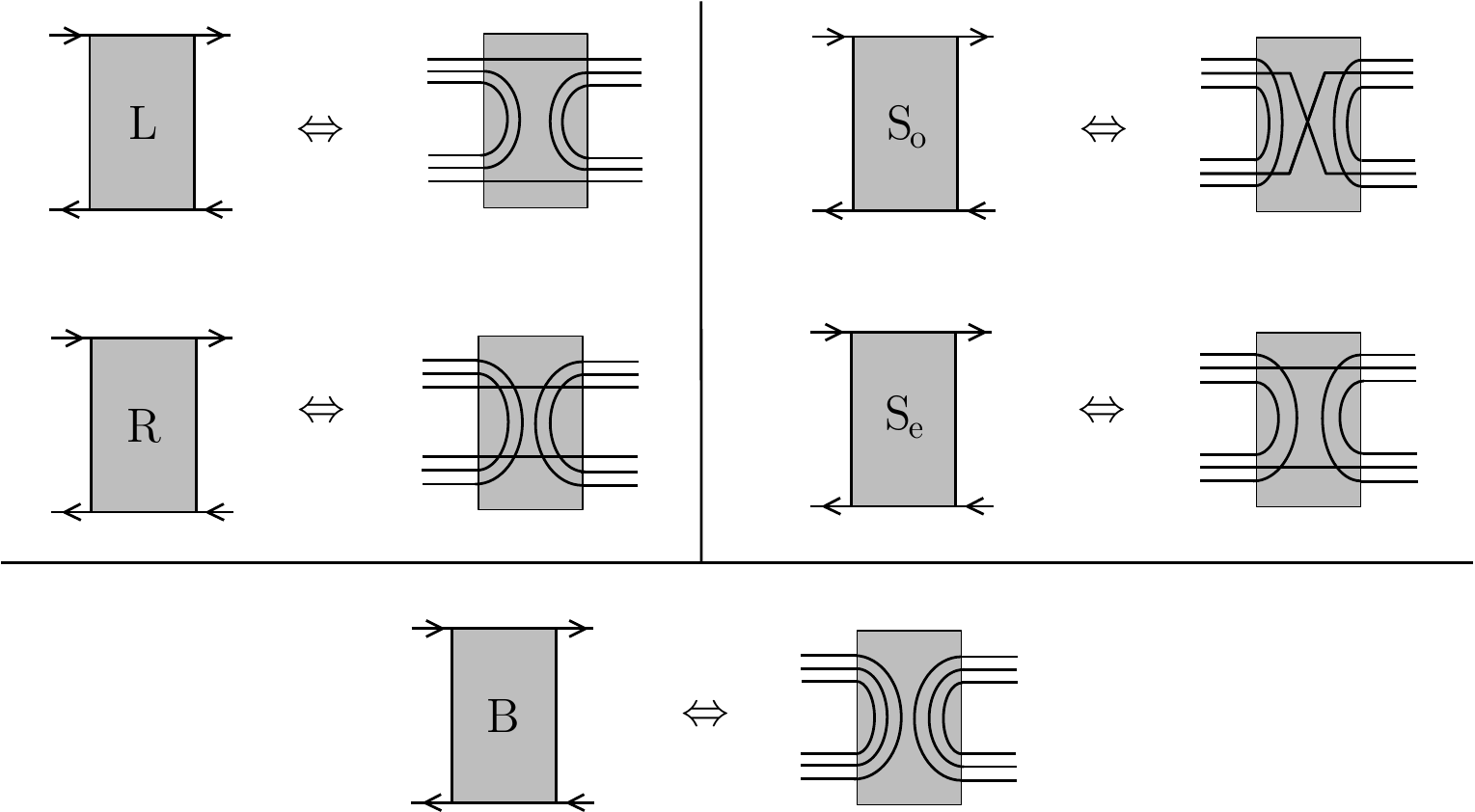}
\end{center}
\caption{The configurations of the strands for each type of chain-vertex.} 
\label{fig:chain_strands}
\end{figure}

Since a scheme is not an MO-graph (due to the presence of chain-vertices), we need
to extend the definition and main properties seen so far in order to allow for chain-vertices.   
In the following,  an \emph{MO-graph with chain-vertices} 
is an oriented 4-regular map with two types of vertices: standard vertices and chain-vertices
(each chain-vertex being labelled by either $\{L,R,S_e,S_o,B\}$),
so as to satisfy the local rules of Figure~\ref{fig:rules_MO_graph_chain}.   
Such a graph is possibly rooted, i.e., has a fake-vertex of degree $2$ in the middle
of some edge.   Note 
that the class of (rooted) 
MO-graph with chain-vertices is \emph{larger} than the class of schemes 
 (we will see below which rooted MO-graphs with chain-vertices are schemes). 
In order to compute 
the number of faces of an MO-graph  $\tilde{G}$ with chain-vertices, 
we need to specify locally the \emph{strand structure} at every type of chain-vertex; as for MO-graphs, we imagine here that each edge is turned into $3$ parallel strands
and we have to specify at each chain-vertex how the incident strands go through the vertex. 
The specification is given by Figure~\ref{fig:chain_strands}; in the case of an unbroken chain-vertex
the two strands that go through the chain-vertex are called the \emph{crossing strands} at
that chain-vertex.  Note that this specification gives the natural strand structure to expect
whenever a chain-vertex $w$ is to be consistently substituted by a chain $c$ (for instance a chain-vertex of type $L$ is to be susbtituted by an unbroken chain of dipoles of type $L$), that is, 
the strand structure of $w$ reflects how the strands arriving at $c$ are routed (some
bouncing back, some going through $c$). 
Then, as for classical MO-graphs, 
a \emph{face} is a closed walk formed from strands. 
The \emph{degree} $\delta$ of $\tilde{G}$  
is defined as
\begin{equation}
2\delta=6c+3V-2F+4U+6B,
\end{equation}
where $c$, $V$, $F$ are as usual the numbers of connected components, standard vertices, and faces,  and where $U$ and $B$ stand respectively
for the numbers of unbroken chain-vertices and broken chain-vertices. 

 An edge $e$ of $\tilde{G}$ is said
to be \emph{adjacent} to a chain-vertex if the two half-edges of $e$ are the two half-edges 
on the same side of a chain-vertex of $\tilde{G}$. Then $\tilde{G}$ is said to be \emph{melon-free} if it has no melon nor an edge adjacent to a chain-vertex. It is easy to see that an MO-graph is melon-free
iff its scheme is melon-free. 

In a melon-free (possibly rooted) MO-graph $\tilde{G}$ with chain-vertices, define a \emph{chain} as a sequence
 $d_1\ldots,d_p$ of elements that are either dipoles (not passing by the root if $\tilde{G}$ is rooted) or chain-vertices, 
such that for each $1\leq i<p$, $d_i$ and $d_{i+1}$ are connected by two edges involving two half-edges on the same side of $d_i$ and two half-edges on the same side of $d_{i+1}$, see  Figure~\ref{fig:chain_dipoles_chain_vert}. 
A \emph{proper chain} is a chain of at least two elements. 

\begin{figure}[t!]
\begin{center}
\includegraphics[width=8cm]{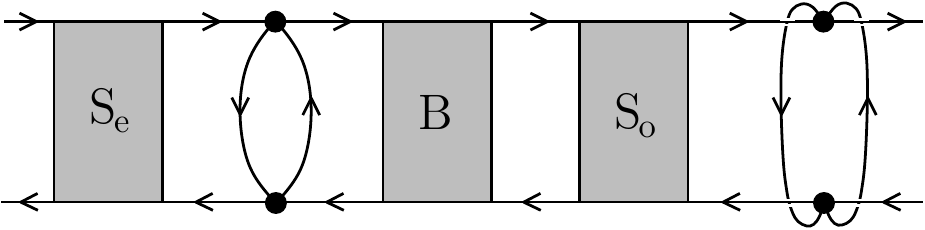}
\end{center}
\caption{A chain of five elements in an MO-graph with chain-vertices: it consists of two dipoles and three chain-vertices.} 
\label{fig:chain_dipoles_chain_vert}
\end{figure}

Now define a \emph{reduced scheme} as a rooted melon-free MO-graph with chain-vertices
and with no proper chain. By construction, the scheme of a rooted melon-free MO-graph (with no chain-vertices) is a reduced scheme. Claim~\ref{claim:disjoint} then easily yields the following bijective statement:

\begin{prop}
Every rooted melon-free MO-graph is uniquely obtained as a reduced scheme where
each chain-vertex is consistently substituted by a chain of at least two dipoles (consistent means that if the chain-vertex is of type $L$, then the substituted chain is an unbroken chain of L-dipoles, etc). 
\end{prop}

The following result ensures that the degree definition for MO-graphs with chain-vertices is consistent with the replacement of chains by chain-vertices:

\begin{lem}\label{lem:substitute}
Let $G$ be an MO-graph with chain-vertices. 
 And let $G'$ be an MO-graph with chain-vertices obtained from $G$
by consistently substituting a  chain-vertex 
by a chain of dipoles. Then the degrees of $G$ and $G'$ are the same.
\end{lem}
\begin{proof}
Let $m$ be the chain-vertex where the substitution takes place,  and let $k$ be the 
number of dipoles in the substituted chain. Denote by $\delta, c,V,F,U,B$ the parameters (degree, numbers of connected components, standard vertices, faces, unbroken chain-vertices, broken chain-vertices) for $G$, and by 
$\delta', c',V',F',U',B'$ the parameters for $G'$. Clearly $c'=c$, $V'=V+2k$.
If $m$ is unbroken, then $U'=U-1$ and $B'=B$.  
Say $m$ is of type $L$ (the arguments for types $R,S_e,S_o$ are similar).  
Then the strand structure remains the same, except for 
 $k$ new faces of type $L$ (of degree $2$), $k-1$ new faces of type $R$ (of degree $4$) 
and $k-1$ new faces of type $S$ (of degree $4$) inside the substituted chain. Hence $F'=F+3k-2$. 
  Consequently, $\delta'=\delta$. 
If $m$ is broken, then $U'=U$ and $B'=B-1$. The strand structure remains the same, except for 
 $k-1$ new faces of each type $L,R,S$ inside the substituted chain (an easy case inspection ensures that substituting a chain of two dipoles of different types
always bring $3$ new faces, one of each type, 
and any additional dipole brings $3$ new faces, one of each type). 
Hence $F'=F+3k-3$.  
  Consequently, $\delta'=\delta$. 
\end{proof}

\begin{coro}
A rooted melon-free MO-graph has the same degree as its scheme. 
\end{coro}

\section{Finiteness of the set of reduced schemes at fixed degree}
The main result in this section is the following:

\begin{prop}\label{prop:finite}
For each $\delta\in\tfrac12\mathbb{Z}_+$, the set of reduced schemes of degree $\delta$
is finite. 
\end{prop}

Similarly as in~\cite{GS}, this result will be obtained from two successive lemmas, proved
respectively in Section~\ref{sec:proof_lem_bound1} and Section~\ref{sec:proof_lem_bound2}:

\begin{lem}\label{lem:bound1}
For each reduced scheme of degree $\delta$, the sum $N(G)$ of the numbers of dipoles
and chain-vertices satisfies $N(G)\leq 7\delta-1$. 
\end{lem}

\begin{lem}\label{lem:bound2}
For $k\geq 1$ and $\delta\in\tfrac12\mathbb{Z}_+$, there is a constant $n_{k,\delta}$ (depending
only on $k$ and $\delta$) such that 
 any connected unrooted MO-graph (without chain-vertices) of degree $\delta$  
with at most $k$ dipoles has at most $n_{k,\delta}$ vertices. 
\end{lem}

\begin{figure}
\begin{center}
\includegraphics[width=12.7cm]{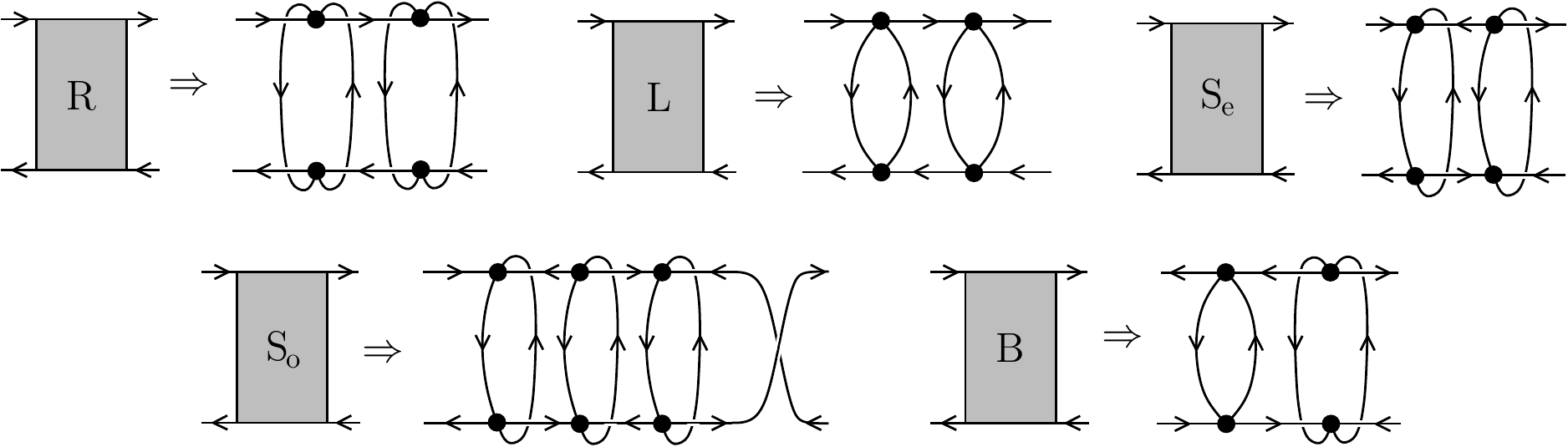}
\end{center}
\caption{The drawings show a choice of canonical substitution of chain-vertices by proper chains that 
preserves the strand structure (hence the degree) and yields
an injective mapping from MO-graphs with chain-vertices to MO-graphs without chain-vertices.}
\label{fig:substitute_canonical}
\end{figure}

Let us show how Lemmas~\ref{lem:bound1} and~\ref{lem:bound2} imply Proposition~\ref{prop:finite}. Let $S$ be a reduced scheme of degree $\delta$, and let $G(S)$
be the rooted MO-graph with no chain-vertices obtained by substituting each chain-vertex of $S$ by a chain of dipoles of length $2$ or $3$ as shown in Figure~\ref{fig:substitute_canonical}. As already seen in Lemma~\ref{lem:substitute}, such substitutions
preserve the degree, and the mapping $S\to G(S)$ is injective. In addition, the number of dipoles of $G(S)$ is at most three times the total number of dipoles and chain-vertices of $S$, hence   
$G(S)$ has at most $3\cdot(7\delta-1)$ dipoles  according to Lemma~\ref{lem:bound1}. 
Unrooting $G(S)$ might increase the number of dipoles by up to $2$ (recall that 
dipoles are not counted in rooted MO-graphs if they pass by the root, and 
at most $2$ dipoles
can pass by the root, the case of $2$ additional dipoles happening 
when the root-edge belongs to a melon).  
Hence, writing $k:=3\cdot(7\delta-1)+3$,  Lemma~\ref{lem:bound2} ensures that $G(S)$
has at most $n_{k,\delta}$ vertices.  
Since there is an injective mapping from reduced schemes of degree $\delta$ to 
rooted MO-graphs of size bounded by the fixed quantity  $n_{k,\delta}$, we conclude that
the number of reduced schemes of degree $\delta$ is finite. 

We finally state the following ``loop-removal'' lemma, which will be useful at some points
in the following:

\begin{figure}
\begin{center}
\includegraphics[width=4cm]{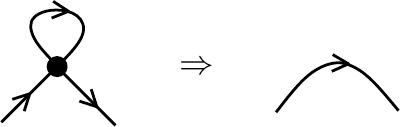}
\end{center}
\caption{Removing a loop in an MO-graph.}
\label{fig:remove_loop}
\end{figure}

\begin{lem}\label{lem:remove_loop}
Let $G$ be an MO-graph of degree $\delta\in\tfrac12\mathbb{Z}_+$,  
let $\ell$ be a loop of $G$, and let $G'$ be the MO-graph obtained
from $G$ by erasing the loop and its incident vertex, as shown in Figure~\ref{fig:remove_loop}. Then
 $G'$ has degree $\delta-1/2$. Hence $G$ has at most $2\delta$ loops. 
\end{lem}
\begin{proof}
Clearly $G'$ has the same number of connected components as $G$, has one vertex less, and has
one face less (which has length $1$).
\end{proof}

\subsection{Analysis of the removal of dipoles and of chain-vertices}\label{sec:analysis_removal}
We analyze here how the degree of an MO-graph with chain-vertices (not  
necessarily a reduced scheme, possibly with melons) 
evolves when removing a chain-vertex. We then have a similar analysis for the removal
of a dipole.  

\begin{figure}
\begin{center}
\includegraphics[width=12.4cm]{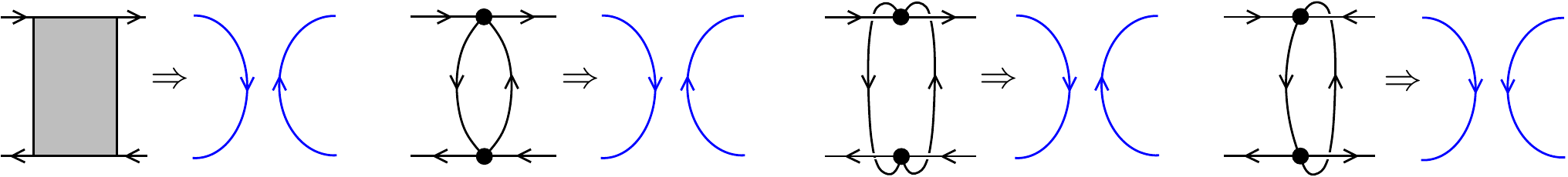}
\end{center}
\caption{The first drawing shows the removal of a chain-vertex (whatever its type). The next drawings show the removal of a dipole in each of the three types (L,R,S).}
\label{fig:deletion}
\end{figure}

Let $G$ be an MO-graph with chain-vertices,  and let $m$ be a chain-vertex of $G$. 
 The \emph{removal} of $m$ consists
of the following operations (see the left part of Figure~\ref{fig:deletion}): 
\begin{enumerate}
\item[(i)] delete $m$ from $G$; 
\item[(ii)] on each side of $m$, connect together the two
detached legs (without creating a new vertex). 
\end{enumerate}
Denote by $G'$ the resulting graph.     
The chain-vertex $m$ is said to be \emph{non-separating} if $G'$ has the same number of connected components as $G$, and \emph{separating} otherwise (in which case $G'$ has one more connected  component).  In the separating case, it might be that one of the two resulting connected components
is the cycle-graph (this happens if two half-edges on one side of $m$ belong to the same edge). 
Let $\delta,c,V,F,U,B$ be the parameters for $G$ (degree, numbers of connected components, standard vertices, faces, unbroken chain-vertices, broken chain-vertices), and
let $\delta',V',F',U',B'$ be the parameters  for $G'$. 

Assume first that $m$ is a broken chain-vertex. We clearly have $V'=V$, $F'=F$ (the strand structure is the same before and after removal), $U'=U$, $B'=B-1$ and $c'=c$ (resp. $c'=c+1$) if $m$ is non-separating (resp. separating). Hence, if $m$ is separating, then
$2\delta'=2\delta+6-6=2\delta$, so that $\delta'=\delta$; and if $m$ is non-separating, then
$2\delta'=2\delta-6$, so that $\delta'=\delta-3$. 

Assume now that $m$ is an unbroken chain-vertex. 
We clearly have $V'=V$, $U'=U-1$, $B'=B$, 
and $c'=c$ (resp. $c'=c+1$) if $m$ is non-separating (resp. separating). 
If the two crossing strands at $m$ belong to a same face, then that face splits into two faces when
removing $m$, hence $F'=F+1$; as opposed to that, if the two crossing strands belong
to different faces, then these two faces are merged when removing $m$, hence $F'=F-1$. 
Moreover, the two crossing strands are easily seen to always be in the same face if $m$
is separating.  We conclude that, if $m$ is separating, then $2\delta'=2\delta+6-2-4=2\delta$, so that
$\delta'=\delta$; and if $m$ is non-separating then $2\delta'=2\delta-2\sigma-4$, where $\sigma\in\{-1,+1\}$, hence either $\delta'=\delta-1$ or $\delta'=\delta-3$.   

\medskip

We now consider the operation of removing a dipole $d$ from $G$, removal which consists
of the following operations (the 3 cases are shown in the right-part of Figure~\ref{fig:deletion}): 
\begin{enumerate}
\item[(i)] delete the two vertices and the two edges of $d$ from $G$; 
 \item[(ii)] on each side of $d$, connect together the two
detached legs (without creating a new vertex). 
\end{enumerate}
Let $G'$ be the resulting graph. Again, 
$d$ is called \emph{non-separating} if $G'$ has as many connected components as $G$, and 
\emph{separating} otherwise.  
Observe that the operation of removing $d$
 is the same as substituting $d$ by an unbroken chain-vertex $m$ 
(of the same type as $d$, so that the substitution does not change the degree according
to Lemma~\ref{lem:substitute}) and then removing $m$. From the discussion on the degree variation  when removing $m$, we conclude that, if $d$ is separating, then the degree is the same in $G'$ as in $G$, and
if $d$ is non-separating then the degree of $G'$ equals the degree of $G$ minus a quantity that is either $1$ or $3$. 
 
This analysis thus leads to the following statement:

\begin{lem}\label{lem:removals}
The degree is unchanged when removing a separating chain-vertex or a separating dipole
(hence, by Claim~\ref{claim:disconnected}, the degree is distributed among the 
resulting components).   
The degree decreases by $3$ when removing a non-separating broken chain-vertex.
The degree decreases by $1$ or $3$ when removing a non-separating unbroken chain-vertex 
or a non-separating dipole.
\end{lem}

\begin{coro}\label{coro:removals}
An MO-graph with chain-vertices and with degree $\delta\in\{0,1/2\}$ 
has no non-separating dipole nor non-separating chain-vertex.
\end{coro}

\subsection{Proof of Lemma~\ref{lem:bound1}}\label{sec:proof_lem_bound1}
We show Lemma~\ref{lem:bound1}, i.e., that for each reduced scheme $G$ of degree $\delta$, the sum $N(G)$ of the numbers of dipoles
and chain-vertices in $G$ satisfies $N(G)\leq 7\delta-1$.  

\begin{figure}
\begin{center}
\includegraphics[width=10cm]{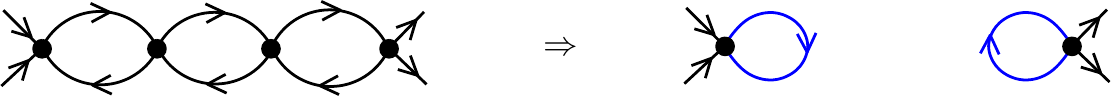}
\end{center}
\caption{In the worst case, removing a dipole might decrease by $3$ the total number of 
(uncolored) dipoles.}
\label{fig:remove_three}
\end{figure}

\begin{figure}
\begin{center}
\includegraphics[width=13cm]{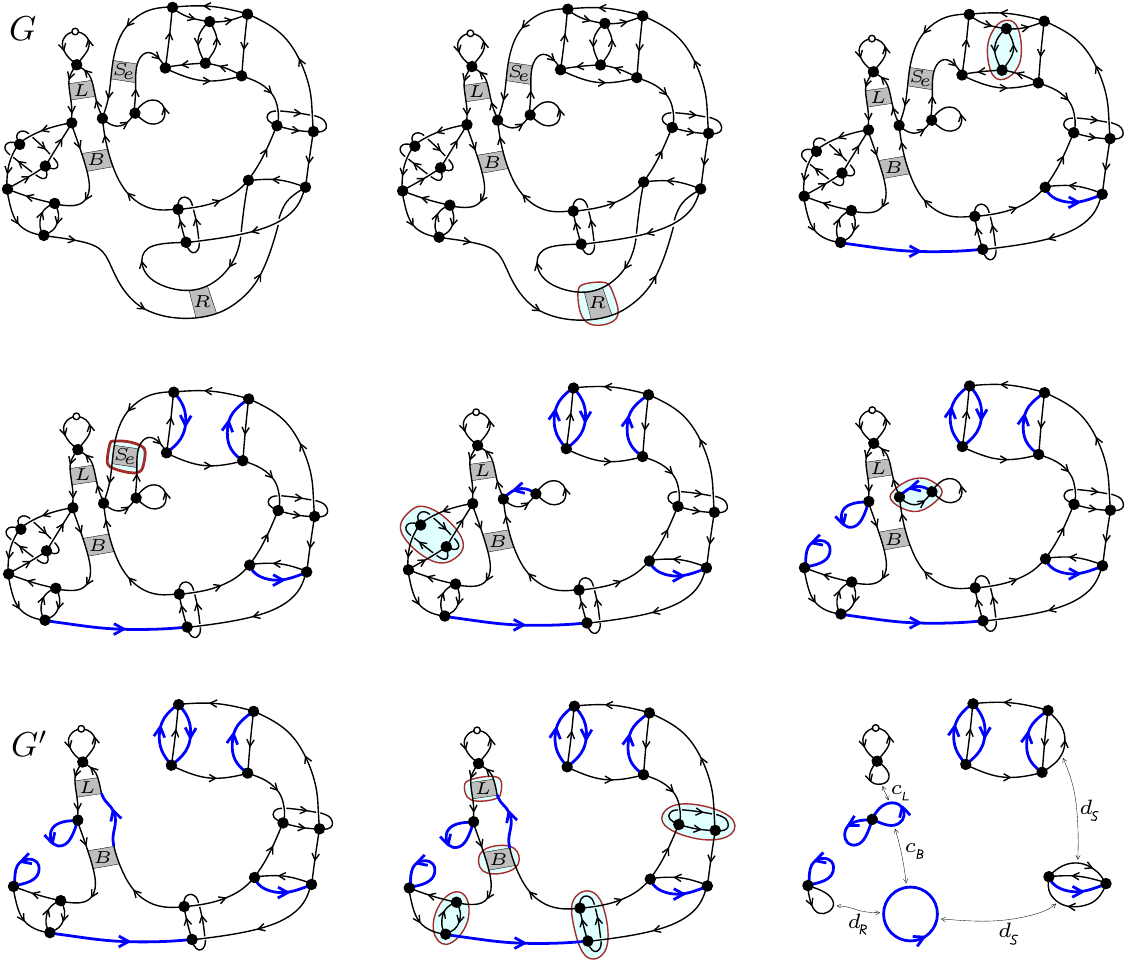}
\end{center}
\caption{The first drawing shows a reduced scheme $G$. Then, iteratively, one removes
at each step a non-separating dipole or a non-separating chain-vertex (at each step the 
non-separating dipole or chain-vertex to be removed next is surrounded). Let $G'$ be the
MO-graph with chain-vertices thus obtained (where colored edges are drawn bolder). 
As the last two drawings show, 
the removal of uncolored dipoles and chain-vertices (which are all separating) of $G'$ yields a tree of components (a tree edge is labelled $c_x$ if it comes from a chain-vertex of type $x$
and is labelled $d_x$ if it comes from an uncolored dipole of type $x$).}
\label{fig:from_scheme_to_tree}
\end{figure}

As a first step, starting from $G$, as long as there is at least one non-separating dipole or one non-separating chain-vertex,  delete it (as in Figure~\ref{fig:deletion}) 
and color the two resulting edges.   
Let $G'$ be the resulting graph with chain-vertices, and let $q$ be the number of removals from $G$ to $G'$. According to Lemma~\ref{lem:removals}, each non-separating removal decreases the degree by at least $1$, hence $q\leq \delta$, $G'$ has degree at most $\delta-q$, and $G'$ 
has at most $2q$ blue edges. 
During the sequence of steps from $G$ to $G'$, denote by $N$ the current 
sum of the numbers of chain-vertices and uncolored dipoles (dipoles with no colored edge).      
It is easy to see that each removal decreases $N$ by at most $3$ (the worst case is shown in Figure~\ref{fig:remove_three}).   
Hence, if we denote by $N(G')$ the value of $N$ for $G'$, we have
$N(G')\geq N(G)-3q\geq N(G)-3\delta$. 

Call a \emph{connection} a chain-vertex or uncolored dipole 
of $G'$. 
Since all  dipoles  
and chain-vertices of $G'$ are separating, $G'$  can be seen as a tree $T$ of ``components'' (see the last row of Figure~\ref{fig:from_scheme_to_tree} for an example), where 
 the edges of $T$ correspond to the $N(G')$ connections of $G'$, 
and the nodes of $T$ correspond to the connected components of $G'$
left after removing the connections 
 (note that some of these connected components might be cycle-graphs, including the component carrying the root). 
For a component $C$ of $T$, the \emph{adjacency order} of $C$
is the number of adjacent components in $T$. 
An edge  $e$ of $C$ resulting from a (removed) connection is called a \emph{marked edge}. 
As shown in Figure~\ref{fig:merger}, 
every marked edge $e$ is involved in $t\geq 1$ connections (hence the adjacency order
of a component is at least its number of marked edges), and results from the merger
of $t+1$ edges of $G'$; $e$ is called \emph{colored} if at least one of these $t+1$ edges
is colored. A component $C$ is called \emph{colored} if  if it bears either 
a colored edge of $G'$ left untouched
by removals, or a colored marked edge; and it is called \emph{uncolored} otherwise.  
Thus, there are at most $2q$ colored components in $G'$. 

\begin{figure}
\begin{center}
\includegraphics[width=12cm]{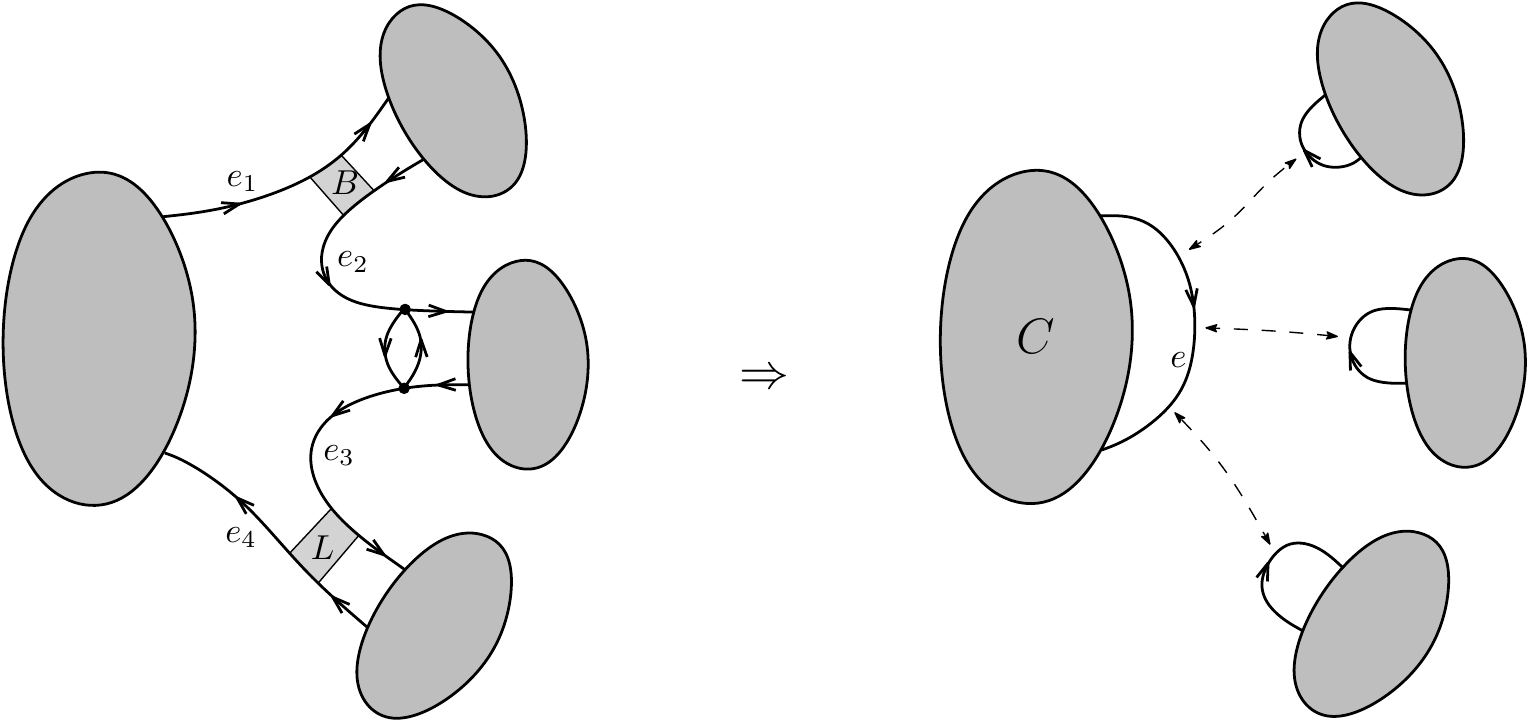}
\end{center}
\caption{For each marked edge $e$ of a component $C$ of $G'$ 
such that $e$ is involved in $t\geq 1$ connections, $e$ results from the merger of $t+1$ edges of $G'$ ($t=3$ in the example, $e$ results from the merger of $\{e_1,e_2,e_3,e_4\}$).} 
\label{fig:merger}
\end{figure}

Similarly as in~\cite{GS}, an important
remark is the following:
\begin{claim}\label{claim:deg_three}
An uncolored component $C$ of $T$ 
of degree $0$ and not containing the root has   adjacency order at least $3$.
\end{claim}
\begin{proof}
The first unrooted MO-graphs of degree $0$ are the cycle-graph (no vertex) and the so-called ``quadruple-edge graph'' (the MO-graph whose underlying map is of genus $0$
 with two vertices connected by $4$ edges). 
And any other MO-graph of degree $0$ can be obtained from the quadruple-edge graph by successive melon insertions. This easily implies that, if $C$ is not the cycle-graph and has adjacency order smaller than $3$ (hence has less than $3$ marked
edges),  
then $C$ has a dipole not passing by any marked edge, a contradiction. 
If $C$ is the cycle-graph and has adjacency order $1$,
 then it clearly yields a melon in $G$, a contradiction.
Finally, if $C$ is the cycle-graph and has adjacency order $2$, then the two edges of $T$ at $C$ 
 ---each of which either corresponds to a separating dipole or to a chain-vertex of $G$--- 
together form a proper
chain (of two elements) of $G$, a contradiction.
\end{proof}
 
Denote by $n_0$ the number of uncolored components in $G'$ of degree $0$ and not containing the
root, $n_0'$ the number 
of other components of degree $0$ (colored or containing the root), and $n_+$ the number of components of positive degree. 
The degrees of the components add up to the degree of $G'$ (recall that, for separating removals,  the degree is distributed among the components), and each component of positive degree has degree at least $1/2$ , hence $n_+/2\leq \delta-q$. 
The components counted by $n_0$ have at least $3$ neighbours in $T$, hence 
$3n_0+n_0'+n_+\leq 2N(G')$. 
Since $T$ has at most $2q$ colored components, 
we have $n_0'\leq 2q+1$ (the $+1$ accounting for 
the component containing the root). And since $T$ is a tree, we have
$$
N(G')=n_0+n_0'+n_+-1,
$$ 
hence, if we eliminate $n_0$ using $3n_0+n_0'+n_+\leq 2N(G')$, we obtain
$$
N(G')\leq \frac{2}{3}(N(G')+n_0'+n_+)-1\leq \frac{2}{3}(N(G')+2q+1+2\delta-2q)-1,
$$
so that $N(G')\leq 4\delta-1$. Since $N(G)\leq N(G')+3\delta$, we conclude that $N(G)\leq 7\delta-1$. 

\subsection{Proof of Lemma~\ref{lem:bound2}}\label{sec:proof_lem_bound2}
We prove here Lemma~\ref{lem:bound2}, i.e., that  
for any $\delta\in\tfrac12\mathbb{Z}_+$ and  
$k\geq 1$, there exists some constant $n_{k,\delta}$ such that any connected unrooted 
MO-graph (with no chain-vertices) of degree $\delta$
with at most $k$ dipoles has at most $n_{k,\delta}$ vertices,  so that 
there are finitely many MO-graphs of degree $\delta$ 
with $k$ dipoles. The bound $n_{k,\delta}$ we obtain is linear in $k$ and $\delta$, with quite large multiplicative constants (we have not pushed to improve the bound, also to avoid making
the proof too complicated).  

The proof is a bit more technical than
the analogous statement for colored graphs in~\cite{GS}, because in the MO model, 
only straight faces have even length, while left and right faces are allowed to have odd lengths. 
Precisely, our proof strategy is to show 
in a series of claims that some properties (e.g. being incident to a dipole)
are satisfied only by a limited (bounded by a quantity depending only on $k$ and $\delta$) number of vertices. Hence if $G$ could be arbitrarily large, it would have vertices not safisfying these properties. As we will see, this would yield the contradiction that $G$ has to be reduced to a certain MO-graph with $6$ vertices. 

In all the proof, $G$ denotes a connected unrooted MO graph with no chain-vertices, of degree $\delta$ and
 with at most $k$ dipoles; $V,E,F$ are
respectively the numbers of vertices, edges, and faces of $G$. Moreover, for $p\geq 1$, $\Fsp$ is 
the number of straight faces of length $2p$. According to Remark~\ref{rk:after_def_dipole}, the only MO-graphs where the set of dipoles
is not equal to the set of faces of length $2$ are the two infinity graphs (both having one vertex). Hence we can assume that 
$G$ is not equal to these two graphs. Thus $F_s^{(1)}\leq k$, which gives:

\begin{claim}\label{claim:pf1}
There are at most $2k$ vertices of $G$ that are incident to a dipole. 
\end{claim}

\vspace{.2cm}

\noindent Also, Lemma~\ref{lem:straight_faces} directly implies 
$$
\sum_{p\geq 3}(p-2)\Fsp\leq 2\delta-2+k.
$$
Since $p-2\geq p/3$ for $p\geq 3$, the total length $L=\sum_{p\geq 3}2p\Fsp$ of straight faces of length larger than $4$ is bounded by  $12\delta-12+6k$, so that we obtain: 

\begin{claim}\label{claim:pf2}
There are at most $12\delta-12+6k$ vertices of $G$ incident to a straight face of length larger than $4$. 
\end{claim} 

A cycle is called \emph{self-intersecting} if it passes twice by a same vertex. 

\begin{figure}
\begin{center}
\includegraphics[width=10cm]{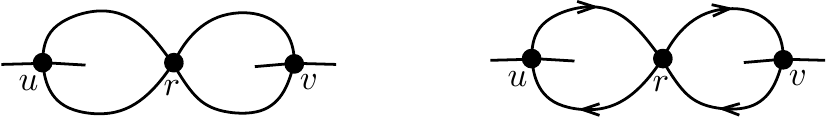}
\end{center}
\caption{On the left the unique map configuration for a self-intersecting straight face of length $4$ 
that does not create a loop.
On the right we see that it is not possible to orient the edges so as to satisfy the local condition
at the vertices $u$ and $v$ if the local condition is satisfied at $r$.}
\label{fig:self_intersect}
\end{figure}

\begin{claim}\label{claim:pf3}
There are at most $2\delta$ self-intersecting straight faces of length $4$ in $G$.  
Hence there are at most $6\delta$ vertices incident to such faces. 
\end{claim}
\begin{proof}
As we see in Figure~\ref{fig:self_intersect}, a self-intersecting straight face of length $4$ must 
contain a loop, which yields the first statement according
to Lemma~\ref{lem:remove_loop} (there are at most $2\delta$ loops). 
The second statement just follows from the fact that these cycles 
 (due to the 
self-intersection) have at most $3$ vertices. 
\end{proof}

Given a non-self-intersecting cycle $C$ of $G$, a \emph{chordal path} for $C$ is a path that starts from $C$, ends at $C$ (possibly at the same vertex), 
and is outside of $C$ inbetween. In the context of maps, we call a chordal path \emph{faulty}
if it starts on one side of $C$ and ends on the other side of $C$. A cycle $C$ is called \emph{faulty}
if it has a faulty chordal path, and, for $k\geq 1$, is called \emph{$k$-faulty} if it has a chordal path of length at most $k$. For $C$ a faulty cycle, we denote by $\wC$ the union of $C$ and of a   (canonically chosen) shortest possible faulty chordal path of $C$; $\wC$ is called the \emph{faulty extension} of~$C$.  

\begin{claim}
Let $d:=96768$.  
There are at most $14d\delta$ 
4-faulty cycles of length $4$ in $G$. Hence there are at most $56d\delta$ vertices
incident to such cycles. 
\end{claim}
\begin{proof}
For $C$ a faulty cycle, $\wC$ is clearly a topological minor of $G$ of genus $1$. 
Hence, if we have $s\geq 1$ vertex-disjoint faulty extensions of faulty cycles, these yield a topological
minor of genus $s$. Since the genus of $G$ is bounded by $2\delta$, we conclude that there can not
be more than $2\delta$ vertex-disjoint faulty extensions of faulty cycles. 
Let $\cE$ be the set of 4-faulty cycles of length $4$ of $G$, and let $\widehat{\cE}:=\{\wC,\ C\in\cE\}$
(where we distinguish the cycle from the faulty path in every faulty extension, so that $|\cE|=|\widehat{\cE}|$). 
Any $\wC\in\widehat{\cE}$ has at most $7$ vertices, and an easy
calculation ensures that any given vertex $v$ of $G$ belongs to at most $d:=96768$ elements
of $\widehat{\cE}$ (this can probably be improved, but we just aim at a certain fixed bound).   Hence $\wC$ can intersect at most $7d$ other elements of $\widehat{\cE}$. 
We easily conclude that $|\cE|=|\widehat{\cE}|\leq 14d\delta$    (otherwise one could construct incrementally a subset of $2\delta\!+\!1\ $ vertex-disjoint elements of $\widehat{\cE}$). 
\end{proof}

Let $\Vexc$ be the set of vertices of $G$ that are either incident to a dipole, or incident to a straight face of length larger than $4$, or incident to a self-intersecting straight face of length $4$, 
or incident to a 4-faulty cycle of length $4$. The four claims above give
\begin{equation}
|\Vexc|\leq (18+56d)\delta+8k-12=:h.
\end{equation}  

\begin{figure}
\begin{center}
\includegraphics[width=8cm]{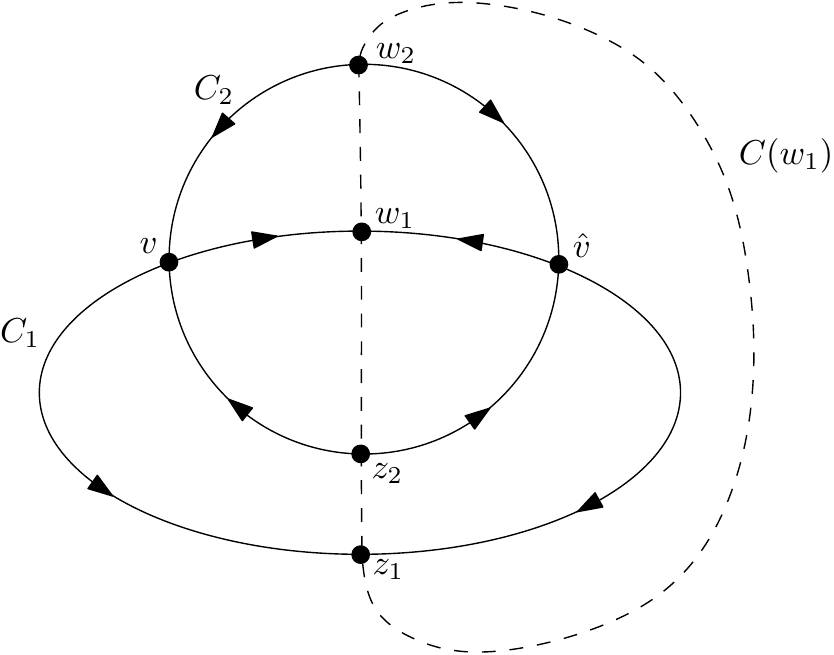}
\end{center}
\caption{The situation in the proof of Lemma~\ref{lem:bound2}}
\label{fig:cycles}
\end{figure}

And let $V_0$ be the set of vertices in $\Vexc$
or adjacent to a vertex in $\Vexc$. Since $G$ is $4$-regular, $|V_0|\leq 5h$. 
We are now going to show that if $G$ has a vertex $v$ not in $V_0$,
then $G$  is actually reduced to having only $6$ vertices. 
Since $v\notin \Vexc$, it is incident to two non-intersecting 
straight faces $C_1$ and $C_2$, both of length $4$. 
Since $C_1$ is not faulty, it has to meet $C_2$ at a vertex different 
from $v$ (otherwise $C_2$ would be a faulty chordal path for $C_1$).   
Let $\hat{v}$ be an intersection of $C_1$ and $C_2$ different from $v$. 
For $i\in\{1,2\}$, let $d_i$ be the distance of $\hat{v}$ from $v$ on $C_i$ 
(note that $d_i\in\{1,2\}$).  
If $d_1=d_2=1$ then either the two corresponding edges $e_1,e_2$  of $C_1$ and $C_2$ form 
an external dipole, or otherwise $e_2$ is a faulty chordal edge for $C_1$. 
Both cases are excluded since $v\notin\Vexc$, 
hence $d_1=d_2=2$. We
write the cycle $C_1$ as $(v,w_1,\hat{v},z_1)$ and the cycle $C_2$ as $(v,w_2,\hat{v},z_2)$,  see Figure~\ref{fig:cycles}.
Let $C(w_1)$ be the straight face (non-intersecting and of length $4$, since $w_1\notin \Vexc$)  passing by $w_1$
and different from $C_1$. 
Since $w_1\notin \Vexc$, $C(w_1)$ is not faulty, hence it 
intersects the 4-cycle $(v,w_1,\hat{v},w_2)$ at a 
vertex different from $w_1$; this other intersection must be at $w_2$ because $v$ and $\hat{v}$
have their $4$ incident edges already in $C_1\cup C_2$. Similarly $C(w_1)$ intersects 
the 4-cycle $(v,w_1,\hat{v},z_2)$ at a vertex different from $w_1$, and this other intersection must
be $z_2$; and $C(w_1)$ intersects 
the 4-cycle $C_1$ at a vertex different from $w_1$, and this other intersection must
be $z_1$, see Figure~\ref{fig:cycles}. Hence 
 $G'=C_1\cup C_2\cup C(w_1)$ forms a connected induced subgraph
that saturates all its vertices (each of the $6$ vertices of $G'$ is incident to $4$ edges of $G'$). 
Hence $G'=G$. In other words, if $G$ has a vertex outside of $V_0$, then $G$ has $6$ vertices.

We can now easily conclude the proof. Define $n_{k,\delta}:=\mathrm{max}(6,5h)$, and assume
$G$ has more than $n_{k,\delta}$ vertices. Since $|V_0|\leq 5h$, $G$ has a vertex not in $V_0$,
which implies that $G$ is reduced to a graph with $6$ vertices, 
giving a contradiction.

\section{MO-graphs of degree $1/2$}\label{sec:schemes_deg_one_half}
In this section we recover the results of~\cite{RT} about the structure of MO-graphs of degree $1/2$,
using ingredients from the preceding sections. 

\begin{prop}\label{prop:deg_one_half}
The only unrooted connected melon-free MO-graphs of degree $1/2$ are the cw and the ccw infinity graph. 
\end{prop}
\begin{proof}
Assume there is a connected melon-free MO-graph $G$ of degree $1/2$ different from the cw or the ccw infinity graph. According to Remark~\ref{rk:after_def_dipole} 
the set of dipoles of $G$ equals the set of its faces of length $2$.  
Note that in degree $1/2$, the genus $g$ has to be $0$ (since $g$ is a nonnegative  integer bounded by $\delta$). 
Hence Lemma~\ref{lem:straight_faces} gives 
$-F_s^{(1)}+\sum_{p\geq 3}(p-2)\Fsp=-1$, so that $F_s^{(1)}\geq 1$, i.e., $G$ has a dipole $d$  (of type $S$). 
Since $\delta<1$, $d$ has to be separating according to Corollary~\ref{coro:removals}. 
Let $G_1$ and $G_2$ 
be the connected components resulting from the removal of $d$. Since $d$ is separating, the degrees of $G_1$ and $G_2$ add up to $1/2$, hence one has degree $0$ and the other has degree $1/2$. 
By convention let  $G_1$ be the one of degree $0$, considered as rooted at the edge resulting from the deletion of $d$. There are two cases: (i) if $G_1$ is the cycle-graph, then 
 $d$ is part of a melon of $G$, giving a contradiction; (ii) if $G_1$ is not the cycle-graph,
then it has a melon (since rooted MO-graphs of degree $0$ are melonic), giving again a contradiction.   
\end{proof}

This yields the following bijective result, where we recover~\cite[Theo.~3.1]{RT}:
\begin{coro}
Every unrooted connected MO-graph of degree $1/2$ is uniquely obtained from the cw or the ccw infinity graph where each of the two edges is substituted by a rooted melonic graph.
\end{coro}
\begin{proof}
All the arguments in the proof of Proposition~\ref{prop:melon_free} (extraction of a unique melon-free core) can be directly recycled  to give the following statement:
``for $\delta\in\tfrac12\mathbb{Z}_+^*$, every unrooted connected MO-graph of degree $\delta$ is
uniquely obtained as an unrooted connected 
melon-free MO-graph of degree $\delta$ where each edge
is substituted by a rooted melonic graph'', which ---together with Proposition~\ref{prop:deg_one_half}--- 
yields the result.
\end{proof}

Another direct consequence of Proposition~\ref{prop:deg_one_half} is the following statement that will prove
useful in the next section:

\begin{coro}\label{coro:dipole12}
Every rooted connected MO-graph of degree $1/2$ different from the (rooted) cw or ccw infinity
graph has a dipole. 
\end{coro}
\begin{proof}
According to Proposition~\ref{prop:deg_one_half}, 
the underlying unrooted MO-graph (obtained by erasing
the fake-vertex of degree $2$) has a melon, hence there is a dipole not including the root edge;
by definition this dipole is also a dipole of the rooted MO-graph (recall that dipoles of rooted MO-graphs are required not to pass by the root). 
\end{proof}

\section{Dominant schemes}\label{sec:dominant}
A reduced scheme $S$ of degree $\delta\in\tfrac12\mathbb{Z}_+$ is called \emph{dominant} if it 
maximizes (over reduced schemes of degree $\delta$) 
the number $b$ of broken chain-vertices  (as we will see in Section~\ref{sec:gen_funct},  
$b$ determines the singularity order of the generating function of rooted MO-graphs of degree $\delta$ and reduced scheme $S$, the larger $b$ the larger the singularity order). 
In this section we show that, in each degree $\delta$,
 the maximal number of broken chain-vertices is $4\delta-1$, and we precisely determine
what are the dominant schemes. 

Let $S$ be a reduced scheme of degree $\delta$, and let $b$ be the number of broken chain-vertices of $S$.  
We now determine a bound for $b$ in terms of $\delta$. The discussion is very similar to the one of
Section~\ref{sec:proof_lem_bound1}, with the difference that only the broken chain-vertices are removed.
First, as long as there is a 
non-separating broken chain-vertex, we remove it (as shown in Figure~\ref{fig:deletion}) and color 
the two resulting edges. Let $q$ be the number
of such removals until all broken chain-vertices are separating; at this stage, denote by $S'$ the resulting MO-graph with chain-vertices, which has degree $\delta'=\delta-3q$ according to Lemma~\ref{lem:removals}, and has at most $2q$ colored edges.  Then $S'$ can be seen as a tree $T$ 
of components that are obtained by removing all separating broken chain-vertices.  
It is also convenient to consider the tree $\hat{T}$ 
of components that are obtained by removing all separating chain-vertices and all separating uncolored dipoles ---dipoles with no colored edge--- of $S'$ (in $T$ and $\hat{T}$ some components
of degree $0$ might be cycle-graphs).  Note that $\hat{T}$ is a refinement of $T$, i.e., each component $C$ of $T$ ``occupies'' a subtree $T_C$ of $\hat{T}$.   
For a component $C$ of $T$ (resp. of $\hat{T}$), the \emph{adjacency order} of $C$
is defined as the number of adjacent components in $T$ (resp. in $\hat{T}$). 
An edge  $e$ of $C$ resulting from a separating removal for $T$ (resp. $\hat{T}$)
is called a \emph{marked edge}. 
Such an edge $e$ results from the merger
of edges of $S'$. 
Then $e$
is called \emph{colored} if at least one of these edges is a colored edge of $S'$.  
And a component $C$ of $T$ (resp. of $\hat{T}$) is called \emph{uncolored} if it bears no colored edge (marked or unmarked).  
Since $S'$ has at most $2q$ colored edges, $T$ (resp. $\hat{T}$) 
has at most $2q$ colored components. 

Consider an uncolored component $C$ of $T$ of degree $0$ and not containing the root. 
Note that $C$ has no non-separating dipoles nor non-separating chain-vertices, by  Corollary~\ref{coro:removals}. 
In addition, by the arguments of Claim~\ref{claim:deg_three}, any component
of $T_C$ must have adjacency order at least $3$, otherwise it would yield a melon or a proper
chain in $S$. This clearly implies that $C$ can not have adjacency order smaller than $3$, and if it has adjacency order $3$, then $T_C$ consists of a unique component.  

Let $n_0$ be the number of components of $T$ of degree $0$, 
uncolored, and not containing the root. Let $n_0'$ be the number of other components of degree $0$. Let $n_+$ be the number of components of positive degree. Moreover, let $p$ be the number
of edges of $T$, which is also the number of broken chain-vertices of $S'$, so that $b=q+p$.

 The degrees of the components add up to the degree $\delta'=\delta-3q$ of $G'$ (recall that, for separating removals,  the degree is distributed among the components), and each component of positive degree has degree at least $1/2$ , hence 
$$n_+/2\leq \delta-3q.$$
The components counted by $n_0$ have at least $3$ neighbours in $T$, hence 
$$3n_0+n_0'+n_+\leq 2p.$$ 
The graph $S'$ has at most $2q$ blue edges, hence 
$$n_0'\leq 2q+1,$$ where the $+1$ accounts for 
the component containing the root. 
Since $T$ is a tree, we have
$$
p=n_0+n_0'+n_+-1,
$$ 
hence, eliminating $n_0$ using $3n_0+n_0'+n_+\leq 2p$, we obtain
$$
p\leq \frac{2}{3}(p+n_0'+n_+)-1\leq \frac{2}{3}(p+2q+1+2\delta-6q)-1,
$$
so that $p\leq 4\delta-8q-1$. 
Hence $b=p+q\leq 4\delta-7q-1\leq 4\delta-1$. 

If a scheme $S$ is such that 
$b$ reaches the upper bound $4\delta-1$, then it implies that all the above inequalities are tight, 
hence $q=0$ (all broken chain-vertices are separating), $n_0'=1$ (the component
containing the root has degree $0$), $3n_0+n_0'+n_+=2p$ (all the components of positive degree and the component containing the root are leaves of $T$, and the other components of degree $0$ have $3$ neighbours in $T$), and $n_+=2\delta$ (all positive degree components have degree $1/2$). 

We take a closer look at the components of degree $0$. Let $C$ be such a component
not containing the root; note that $C$
 has no non-separating dipole nor non-separating chain-vertex according to Corollary~\ref{coro:removals}. 
Since $C$ has adjacency order $3$ and is uncolored (because $q=0$), 
as already mentioned $T_C$ must have a single component,  hence $C$ 
 has no chain-vertex and any dipole
of $C$ must pass by a marked edge of $C$.  
Recall that an unrooted MO-graph of degree $0$ is either the cycle-graph, or 
can be obtained from the ``quadruple-edge'' MO-graph (two vertices connected by $4$ edges) by 
successive insertions of melons. Note that, as soon as at least one melon is inserted in the 
quadruple-edge MO-graph, there are two vertex-disjoint melons. Hence, an MO-graph of degree $0$
with at most $3$ marked edges and strictly more than $2$ vertices has a dipole that avoids the marked edges. 
It follows that $C$
has to be either the cycle-graph or the quadruple-edge graph. 
Now, if $C$ is the component containing the root, since it has only one adjacent component $C'$, 
it has to be the (rooted) cycle-graph. Indeed, 
the case of a rooted quadruple-edge MO-graph is excluded; in that case, among the 4 edges of $C$, 
one carries the root, one is marked,  
and the two other ones form a dipole that, together
with the chain-vertex connecting $C$ to $C'$, would form a proper chain of $S$, a contradiction. 

We now take a closer look at components of positive degrees (of degree $1/2$). Let  
 $C$ be such a component (which has a unique marked edge since it is at a non-root leaf of $T$). Since $C$ has degree $1/2$ it can not have non-separating chain-vertices
or non-separating unbroken chain-vertices, according to Corollary~\ref{coro:removals}. And 
 by the arguments of Claim~\ref{claim:deg_three},   
 components of degree $0$ of $T_C$ must have adjacency order at least $3$. 
Since the degrees of components in $T_C$ add up to $1/2$, 
this easily implies that the unique possibility is $T_C$ having a single node. Hence $C$
has no dipole avoiding the marked edge nor chain-vertices, so that $C$ must be the cw or the ccw infinity graph, 
according to Corollary~\ref{coro:dipole12}.

\begin{figure}
\begin{center}
\includegraphics[width=8cm]{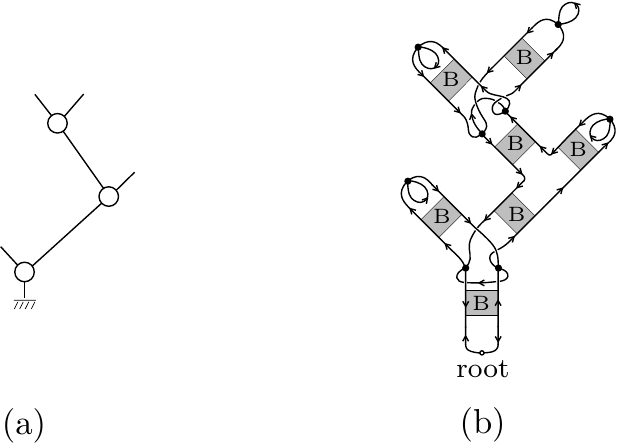}
\end{center}
\caption{(a) A rooted binary tree $\beta$ with $5$ leaves. (b) 
One of the $2^{10}$ dominant schemes of degree $2$ arising from $\beta$.}
\label{fig:dominant}
\end{figure} 

\begin{prop}\label{prop:structure_dominant}
For $\delta\in\tfrac12\mathbb{Z}_+^*$, the dominant schemes of degree $\delta$ arise from
 rooted binary trees (see Figure~\ref{fig:dominant} for an example) 
with $2\delta+1$ leaves, $2\delta-1$ inner nodes, and $4\delta-1$ edges,
where the root-leaf is occupied by the rooted cycle-graph, the $2\delta$ other leaves are occupied
by (cw or ccw) infinity graphs, the $2\delta-1$ inner nodes are occupied either
by the cycle-graph or by  the quadruple-edge graph,  
and the $4\delta-1$ edges are occupied by separating broken chain-vertices.

Each rooted binary tree with $2\delta+1$ leaves yields $2^{6\delta-2}$ dominant schemes. 
\end{prop}
\begin{proof}
In the analysis above we have seen that $b\leq 4\delta-1$ and that any scheme with $b=4\delta-1$ must be of the stated form. 
Conversely any scheme of the stated form  is a valid scheme (no melon nor proper chain) of degree $\delta$ with $b=4\delta-1$ chain-vertices. 
Hence these schemes are exactly the dominant schemes at degree $\delta$. 
Each rooted binary tree with $2\delta+1$ leaves gives rise to $4^{2\delta-1}2^{2\delta}=2^{6\delta-2}$ dominant schemes. 
Indeed, at each inner node, one has to decide whether it is occupied by the cycle-graph or
by the quadruple-edge graph, and in the second case, one has to choose among the three possible configurations for the free edge (the unique edge not involved in any connection) since the straight face passing
by the free edge can go either toward the left child, or the right child, or  the parent. 
Finally, one has to decide at each non-root leaf if it is occupied by the cw or the ccw infinity graph. 
\end{proof}

\begin{rk}
As we have recalled in Section~\ref{sec:regular_subfamily}, regular colored graphs in dimension $3$ naturally form a subfamily of MO-graphs. However, since half-integer degrees are not possible in the model of  regular colored graphs, the dominant schemes differ; as shown in~\cite{GS}, in degree $\delta\in\mathbb{Z}_+$,
the dominant schemes are associated to rooted binary trees with $\delta+1$ leaves (and $\delta-1$ inner nodes), where the root-leaf is occupied by a root-melon, while the $\delta$ non-root leaves
are occupied by the unique scheme of degree $1$. 
\end{rk}

\section{Generating functions and asymptotic enumeration}\label{sec:gen_funct}

Let $\delta\in\tfrac12\mathbb{Z}_+$, and let $\cS_{\delta}$ be the (finite) set of reduced schemes
of degree $\delta$. For each $S\in\cS_{\delta}$, let $G_S^{(\delta)}(u)$ be the generating function of 
 rooted melon-free MO-graphs of reduced scheme $S$, where $u$ marks half the number of
non-root vertices (i.e., for $p\in\tfrac12\mathbb{Z}_+$, a weight $u^p$ is given to rooted MO graphs with $2p$ non-root vertices). Let $p$ be half the number of non-root standard vertices of $S$, $b$
the number of broken chain-vertices, $a$ the number of unbroken chain-vertices of type $L$ or $R$,
$s_e$ the number of even straight chain-vertices, and $s_o$ the number of odd straight chain-vertices. 
The generating functions for unbroken chains of type L (resp. R) is clearly $u^2/(1-u)$,
the one for even straight chains is $u^2/(1-u^2)$, the one for odd straight chains is $u^3/(1-u^2)$,
and the one for broken chains is $(3u)^2/(1-3u)-3u^2/(1-u)=6u^2/((1-3u)(1-u))$. Therefore
$$
G_S^{(\delta)}(u)=u^p\frac{u^{2a}}{(1-u)^a}\frac{u^{2s_e}}{(1-u^2)^{s_e}}\frac{u^{3s_o}}{(1-u^2)^{s_o}}\frac{6^bu^{2b}}{(1-3u)^b(1-u)^b}.
$$
Denoting by $c$ the total number of chain-vertices and by $s=s_e+s_o$ the total number of straight chain-vertices, this simplifies as
\begin{equation}
G_S^{(\delta)}=\frac{6^bu^{p+2c+s_o}}{(1-u)^{c-s}(1-u^2)^s(1-3u)^b}.
\end{equation}
It is well-known that the generating function of rooted melonic graphs is given by 
$$
T(z)=1+zT(z)^4.
$$
In addition, a rooted melon-free MO-graph with $2p$ non-root vertices has $4p+1$ edges (recall
that the root-edge is split into two edges) where one can insert a rooted 
melonic graph. Therefore, defining $U(z):=zT(z)^4=T(z)-1$,
 the generating function $F_S^{(\delta)}(z)$ of rooted MO-graphs
of reduced scheme $S$ is given by 
\begin{equation}
F_S^{(\delta)}(z)=T(z)\frac{6^bU(z)^{p+2c+s_o}}{(1-U(z))^{c-s}(1-U(z)^2)^s(1-3U(z))^b},
\end{equation}
and the generating function $F^{(\delta)}(z)$ of rooted MO-graphs of degree $\delta$ is simply 
given by 
\begin{equation}
F^{(\delta)}(z)=\sum_{S\in\cS_{\delta}}F_S^{(\delta)}(z). 
\end{equation}
As discussed in~\cite{GS},
$T(z)$ has its main singularity at $z_0:=3^3/2^8$, $T(z_0)=4/3$, and $1-3U(z)\sim_{z\to z_0}2^{3/2}3^{-1/2}(1-z/z_0)^{1/2}$. Therefore, the dominant terms are those for which 
$b$ is maximized. As shown in Section~\ref{sec:dominant}, these schemes are naturally associated to rooted binary trees with $2\delta-1$ inner nodes, as stated in Proposition~\ref{prop:structure_dominant}.    
For a fixed rooted binary tree $\gamma$ with $2\delta-1$ inner nodes, 
the total contribution of schemes arising from $\gamma$ to the generating function of 
rooted melon-free MO-graphs of degree $\delta$ is 
$$
\frac{(2u^{1/2})^{2\delta}(1+3u)^{2\delta-1}6^{4\delta-1}u^{8\delta-2}}{(1-u)^{4\delta-1}(1-3u)^{4\delta-1}}.
$$
Indeed, such schemes have $b=4\delta-1$, $c=b$, $s=s_o=0$, each of the $2\delta$
non-root leaves is occupied either by the cw or ccw infinity graph, and each of the $2\delta-1$ inner nodes is either occupied by the cycle-graph or 
the quadruple-edge graph that has $3$ possible configurations.

Hence the total contribution to the generating function of rooted MO-graphs is 
\begin{eqnarray*}
&&T(z)\frac{U(z)^{\delta}2^{2\delta}(1+3U(z))^{2\delta-1}6^{4\delta-1}U(z)^{8\delta-2}}{(1-U(z))^{4\delta-1}(1-3U(z))^{4\delta-1}}\\
&=&z^{\delta}T(z)^{1+4\delta}\frac{2^{6\delta-1}3^{4\delta-1}(1+3U(z))^{2\delta-1}U(z)^{8\delta-2}}{(1-U(z))^{4\delta-1}(1-3U(z))^{4\delta-1}}.
\end{eqnarray*}
This is to be multiplied by the Catalan number $\mathrm{Cat}_{2\delta-1}$ of rooted binary
trees with $2\delta-1$ inner nodes. 
Note that, by design, the dominant schemes  
 contribute only to graphs with $2n$ non-root vertices
such that $n\in\delta+\mathbb{Z}$, which is consistent here with  
 the prefactor $z^{\delta}$. Graphs with $2n$ non-root vertices such   that $n\in\delta+1/2+\mathbb{Z}$ have schemes with strictly less than $4\delta-1$ broken chain-vertices,
hence have a singular behaviour of type $(1-z/z_0)^{-2\delta+1+\alpha}$, with $\alpha\in\tfrac12\mathbb{Z}_+$. 
After elementary calculations and applying transfer theorems of analytic combinatorics~\cite{fs}, we obtain: 

\begin{prop}\label{prop:asympt}
For $\delta$ and $n$ in $\tfrac12\mathbb{Z}_+$, let $a_n^{(\delta)}$ be the number of rooted MO-graphs with $2n$ vertices and degree $\delta$. Then, $\delta$ being fixed, for $n\in \delta+\mathbb{Z}$ (and $\Gamma(\cdot)$ denoting the Euler gamma function), 
\begin{equation}\label{eq:asympt_degree}
a_n^{(\delta)}\sim\ \mathrm{Cat}_{2\delta-1}\cdot\frac{3^{\delta-3/2}}{2^{2\delta-5/2}}\cdot\frac{n^{2\delta-3/2}}{\Gamma(2\delta-1/2)}\cdot(2^8/3^3)^n\ \mathrm{as}\ n\to\infty.
\end{equation}
and $a_{n+1/2}^{\delta}=O(a_{n}^{\delta}/\sqrt{n})\ \mathrm{as}\ n\to\infty$. 
\end{prop}

\begin{rk}
According to Proposition~\ref{prop:asympt}, for $\delta\in\mathbb{Z}_+$   $a_n^{(\delta)}$  is asymptotically $\Theta(n^{2\delta-3/2}(2^8/3^3)^n)$, where the multiplicative
constant involves $\mathrm{Cat}_{2\delta-1}$. 
As a comparison, as shown in~\cite{GS}, for $\delta\in\mathbb{Z}_+$ the number $c_n^{(\delta)}$ of rooted colored graphs
(in dimension $3$) of degree $\delta$ with $2n$ vertices is asymptotically $\Theta(n^{\delta-3/2}(2^8/3^3)^n)$, where the multiplicative constant involves $\mathrm{Cat}_{\delta-1}$. We also deduce from these estimates that for fixed $\delta\in\mathbb{Z}_+$, the probability that a random rooted MO-graph of degree $\delta$ with $2n$ vertices (where $n\in\mathbb{Z}_+$) is a regular colored graph is  $\Theta(n^{-\delta})$. 
\end{rk}

Now, going back to Remark~\ref{rk:doodles}, 
we can also easily obtain the asymptotic enumeration under
the planarity constraint, based on the following:

\begin{lem}
Each rooted melon-free MO-graph whose reduced scheme is dominant is planar. 
\end{lem}
\begin{proof}
As shown in Section~\ref{sec:dominant}, a rooted melon-free MO-graph $G$ of degree $\delta\in\tfrac12\mathbb{Z}$ has a total of $2\delta$ loops (one at each non-root leaf
of the associated binary tree, see Figure~\ref{fig:dominant}). According to Lemma~\ref{lem:remove_loop},  
erasing all these $2\delta$ loops yields a (rooted) MO-graph $G'$ of degree $0$, hence
a planar MO-graph. Clearly the $2\delta$ erased loops can be inserted back without breaking planarity, from which we conclude that $G$ is planar. 
\end{proof}

\begin{figure}
\begin{center}
\includegraphics[width=8cm]{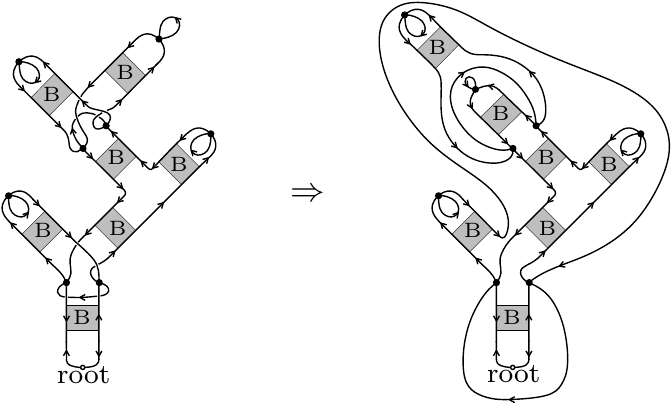}
\end{center}
\caption{Planar redrawing of the scheme of Figure~\ref{fig:dominant}.}
\label{fig:planar_redraw}
\end{figure}

\begin{rk}
Given any dominant scheme such as shown in Figure~\ref{fig:dominant} 
(naturally drawn following the shape
of a rooted binary tree), 
one can easily figure out how to retract/extend some parts of the drawing so as to obtain
a  planar redrawing, see Figure~\ref{fig:planar_redraw}.
\end{rk}

\begin{coro}\label{coro:almost_sure}
For each fixed $\delta\in\tfrac12\mathbb{Z}_+$ and for $n\in\delta+\mathbb{Z}$, the probability
that a rooted MO-graph of degree $\delta$ with $2n$ vertices is planar tends to $1$ as $n\to\infty$. 
\end{coro}
\begin{proof}
It just follows from the observation that the edge-substitution by melonic components clearly
preserves planarity, so that all rooted MO-graphs arising from a dominating reduced scheme  are planar. This concludes the proof since these rooted MO-graphs are the ones that dominate the asymptotic expansion.
\end{proof}

According to Remark~\ref{rk:doodles}, in the planar case, rooted MO-graphs correspond (bijectively) to 
rooted 4-regular maps, the straight faces of the MO-graph identify to the knot-components 
of the map, and the number of straight faces is equal to $V/2+1-\delta$, with $V$ the number of vertices and $\delta$ the degree. Under this rephrasing,  Proposition~\ref{prop:asympt} and Corollary~\ref{coro:almost_sure}
yield the following:

\begin{prop}
For $n\in\tfrac12\mathbb{Z}_+$ and $k\in\mathbb{Z}_+$, let $b_{n}^{(k)}$ be the number of rooted 4-regular planar maps with $2n$ vertices and $k$ knot-components. Then $b_{n}^{(k)}=0$ for $k>n+1$.
In addition, for each fixed $\delta\in\tfrac12\mathbb{Z}_+$, and for $n\in\delta+\mathbb{Z}$, $b_{n}^{(n+1-\delta)}$ has the same asymptotic estimate as $a_n^{(\delta)}$, 
given by ~\eqref{eq:asympt_degree}. 
\end{prop}

\section{Concluding remarks and perspectives}

In this article we have shown that, similarly as in the colored model~\cite{GS},
 the combinatorial study of  MO-graphs can be done by extraction of so-called \emph{schemes},
such that there are \emph{finitely many} schemes in each fixed degree $\delta\in\tfrac12\mathbb{Z}_+$. We have also identified the dominant schemes
in each degree, whose shapes are naturally associated to rooted binary trees. 
Having determined the dominant schemes we have obtained 
 an explicit asymptotic estimate
for the number of (rooted) MO-graphs of fixed degree as the number of vertices tends to infinity.

As already mentioned in the introduction, a first perspective for future work is the implementation of the double scaling limit for the MO tensor model. For the sake of completeness, let us mention that the double scaling mechanism was implemented, for a different type of colored model, in \cite{DGR}, using quantum field theoretical-inspired methods (the so-called intermediate field method). From a probabilistic point a view, a different perspective for future work is to investigate whether the new dominant schemes we have exhibited in this paper correspond or not to phases different from the one of branched polymers (it was recently proved that the melon graphs correspond, from this point of view, to branched polymers \cite{GR2}). 
Another appealing perspective is to extend the MO model (currently only developed in dimension $D=3$)  and results of this paper to higher dimensions. Finally, one can address the issue of proving counting theorems for 3D maps using tensor integral techniques, generalizing the fact that counting theorems for maps can be obtained using matrix integral techniques.

\section*{Acknowledgements}
 
The authors acknowledge R\u azvan Gur\u au and Gilles Schaeffer for very stimulating and instructive discussions.
Adrian Tanasa is partially supported by the grants ANR JCJC ``CombPhysMat2Tens" and PN 09 37 01 02. \'Eric Fusy is partially supported by the ANR grant 
``Cartaplus'' 12-JS02-001-01, and the ANR grant ``EGOS'' 12-JS02-002-01. The authors also acknowledge the Erwin Schr\"odinger International Institute for the valuable work environment provided to them during the ``Combinatorics, Geometry and Physics" programme.

\bigskip

\noindent \'Eric Fusy\\
{\it\small  LIX, CNRS UMR 7161, \'Ecole Polytechnique, 91120 Palaiseau, France, EU}

\medskip

\noindent
Adrian Tanasa\\
{\it\small LIPN, Institut Galil\'ee, CNRS UMR 7030, 
Universit\'e Paris 13, Sorbonne Paris Cit\'e,}\\{\it\small 99 av. Clement, 93430 Villetaneuse, France, EU}\\
{\it\small Horia Hulubei National Institute for Physics and Nuclear Engineering,
P.O.B. MG-6, 077125 Magurele, Romania, EU}
\end{document}